\documentclass[reqno, final,a4paper]{amsart}
\usepackage{color}
\usepackage[colorlinks=true,allcolors=blue
]{hyperref}
\usepackage{amsmath, amssymb, amsthm}
\usepackage{mathrsfs}
\usepackage{mathtools}
\usepackage[noabbrev,capitalize,nameinlink]{cleveref}
\crefname{equation}{}{}
\usepackage{fullpage}
\usepackage[noadjust]{cite}
\usepackage{graphics}
\usepackage{pifont}
\usepackage{tikz}
\usepackage{bbm}
\usepackage[T1]{fontenc}
\usetikzlibrary{arrows.meta}

\usepackage{environ}
\usepackage{framed}
\usepackage{url}
\usepackage[linesnumbered,ruled,vlined]{algorithm2e}
\usepackage[noend]{algpseudocode}
\usepackage[labelfont=bf]{caption}
\usepackage{framed}
\usepackage[framemethod=tikz]{mdframed}
\usepackage{pgfplots}
\usepackage{appendix}
\usepackage{graphicx}
\usepackage[textsize=tiny]{todonotes}
\usepackage{tcolorbox}
\usepackage{xcolor}
\usepackage[shortlabels]{enumitem}
\crefformat{enumi}{#2#1#3}
\crefrangeformat{enumi}{#3#1#4 to~#5#2#6}
\crefmultiformat{enumi}{#2#1#3}
{ and~#2#1#3}{, #2#1#3}{ and~#2#1#3}
\allowdisplaybreaks

\let\originalleft\left
\let\originalright\right
\renewcommand{\left}{\mathopen{}\mathclose\bgroup\originalleft}
\renewcommand{\right}{\aftergroup\egroup\originalright}

\newcommand*{\claimproofname}{Proof of claim}
\newenvironment{claimproof}[1][\claimproofname]{\begin{proof}[#1]}{\end{proof}}

\crefname{algocf}{Algorithm}{Algorithms}

\crefname{equation}{}{} 
\AtBeginEnvironment{appendices}{\crefalias{section}{appendix}} 

\usepackage[color]{showkeys} 

\colorlet{refkey}{orange!20}
\colorlet{labelkey}{blue!30}

\crefname{algocf}{Algorithm}{Algorithms}

\numberwithin{equation}{section}
\newtheorem{theorem}{Theorem}[section]
\newtheorem{proposition}[theorem]{Proposition}
\newtheorem{lemma}[theorem]{Lemma}

\crefname{claim}{Claim}{Claims}
\crefname{theorem}{Theorem}{Theorems}

\newtheorem{corollary}[theorem]{Corollary}

\newtheorem*{claim*}{Claim}

\crefname{conjecture}{Conjecture}{Conjectures}

\newtheorem{definition}[theorem]{Definition}

\newtheorem*{definition*}{Definition}

\theoremstyle{definition}
\theoremstyle{remark}
\newtheorem{remark}[theorem]{Remark}

\newcommand{\floor}[1]{\left\lfloor #1 \right\rfloor}

\newcommand{\one}{\mathbbm{1}}

\newcommand{\mb}{\mathbb}

\newcommand{\mbm}{\mathbbm}
\newcommand{\mc}{\mathcal}

\newcommand{\mr}{\mathrm}

\newcommand{\on}{\operatorname}

\newcommand{\wh}{\widehat}

\newcommand{\pvec}[1]{\vec{#1}\mkern2mu\vphantom{#1}'}

\newcommand{\eps}{\varepsilon}
\renewcommand{\Pr}{\mb P}

\allowdisplaybreaks

\title{The edge-statistics conjecture for hypergraphs}

\author[Jain]{Vishesh Jain}
\address{Department of Mathematics, Statistics, and Computer Science, University of Illinois Chicago, Chicago, IL, 60607 USA}
\email{visheshj@uic.edu}

\author[Kwan]{Matthew Kwan}

\address{Institute of Science and Technology Austria (ISTA). Am Campus 1, 3400 Klosterneuburg, Austria}
\email{matthew.kwan@ist.ac.at}

\author[Mubayi]{Dhruv Mubayi}
\address{Department of Mathematics, Statistics, and Computer Science, University of Illinois Chicago, Chicago, IL, 60607 USA}
\email{mubayi@uic.edu}

\author[Tran]{Tuan Tran}

\address{School of Mathematical Sciences, University of Science and Technology of China, Hefei, 230026 Anhui, China}
\email{trantuan@ustc.edu.cn}
\thanks{
Jain was supported by NSF CAREER award DMS-2237646.
Kwan was supported by ERC Starting Grant ``RANDSTRUCT'' No.~101076777. Mubayi was supported by NSF grant DMS-2153576. Tran was supported by the National Key Research and Development Program of China 2023YFA101020.
}

\begin{document}

\maketitle
\begin{abstract}
Let $r,k,\ell$ be integers such that $0\le\ell\le\binom{k}{r}$. Given a large $r$-uniform hypergraph $G$, we consider
the fraction of $k$-vertex subsets which span exactly $\ell$ edges.
If $\ell$ is $0$ or $\binom{k}{r}$, this fraction can be exactly
$1$ (by taking $G$ to be empty or complete), but for all other values
of $\ell$, one might suspect 
that this fraction is always significantly
smaller than $1$.

In this paper we prove an essentially optimal result along these lines:
if $\ell$ is not $0$ or $\binom{k}{r}$, then this fraction
is at most $(1/e) + \eps$, assuming $k$ is sufficiently large
in terms of $r$ and $\varepsilon>0$, and $G$ is sufficiently large
in terms of $k$. Previously, this was only known for a very limited range of values of $r,k,\ell$ (due to Kwan--Sudakov--Tran, Fox--Sauermann, and Martinsson--Mousset--Noever--Truji\'{c}).
Our result answers a question of Alon--Hefetz--Krivelevich--Tyomkyn, who suggested this as a hypergraph generalisation of their \emph{edge-statistics conjecture}. We also prove
a much stronger bound when $\ell$ is far from 0 and $\binom{k}{r}$.
\end{abstract}
\section{Introduction}
Given a $k$-vertex graph $H$, what is the maximum possible number
of $k$-vertex subsets of an $n$-vertex graph that induce a copy
of $H$? Denote this number by $N(n,H)$, so we have $0\le N(n,H)\le\binom{n}{k}$.
A simple averaging argument shows that $N(n,H)/\binom{n}{k}$ is nonincreasing 
in $n$ (for $n\ge k$), so we can define
\begin{equation}
    \operatorname{ind}(H)=\lim_{n\to\infty}\frac{N(n,H)}{\binom{n}{k}}.\label{eq:ind(H)}
\end{equation}
This quantity is called the \emph{inducibility} or the \emph{maximum
induced density} of $H$. It was first considered in 1975 by Pippenger
and Golumbic~\cite{PG75}, and has been studied intensively in the intervening
decades.

In general, it is very difficult to determine $\operatorname{ind}(H)$,
even for small graphs $H$ (for example, the inducibility of the 4-vertex
path is still unknown). However, we do know the \emph{minimum} of
$\operatorname{ind}(H)$, among all $k$-vertex graphs $H$ (provided
$k$ is sufficiently large): indeed, Pippenger and Golumbic~\cite{PG75} showed
that if $H$ has $k$ vertices then 
\begin{equation}
\operatorname{ind}(H)\ge\frac{k!}{k^{k}-k},\label{eq:PG}
\end{equation}
and it is known (see \cite{BHLF, FSW21}) that if $k=5$ or if $k$ is sufficiently large, there is a choice of $H$ that attains this bound.

What about the \emph{maximum} value of $\operatorname{ind}(H)$? It
is easy to see that $\operatorname{ind}(H)=1$ when $H$ is a complete
or empty graph; if we exclude these ``trivial'' examples then
we arrive at the so-called \emph{large inducibility conjecture} of
Alon, Hefetz, Krivelevich and Tyomkyn~\cite[Conjecture 1.2]{AHKT20}. Namely, they identified
several (nontrivial) infinite classes of graphs $H$ with $\operatorname{ind}(H)>1/e$,
and conjectured that the maximum of $\operatorname{ind}(H)$ over
all nontrivial $k$-vertex graphs $H$ tends to $1/e$ as $k\to\infty$.

Alon, Hefetz, Krivelevich and Tyomkyn also made a second much stronger
conjecture called the \emph{edge-statistics
conjecture}~\cite[Conjecture 1.1]{AHKT20}, concerning
a much looser variant of graph inducibility. Specifically, for $0\le\ell\le\binom{k}{2}$,
let $N_2(n,k,\ell)$ be the maximum possible number of $k$-vertex subsets
of an $n$-vertex graph which induce exactly $\ell$ edges, and let
\begin{equation}
    \operatorname{ind}_2(k,\ell)=\lim_{n\to\infty}\frac{N_2(n,k,\ell)}{\binom{n}{k}}.\label{eq:ind(k ell)}
\end{equation}
For each value of $k$, say that $0$ and $\binom{k}{2}$ are the ``trivial'' values of $\ell$; the edge-statistics conjecture says that the maximum of $\operatorname{ind}_2(k,\ell)$ over all nontrivial $\ell$ tends to $1/e$ as $k\to\infty$.  
This conjecture has significance
beyond its consequences for graph inducibility: it can be interpreted
as giving a limit on ``how uniform'' a graph can be, with respect
to statistics of edges in small subsets.

It is equally natural to consider the large inducibility and edge-statistics conjectures for \emph{hypergraphs}; these generalisations were actually explicitly suggested in the same paper of Alon, Hefetz, Krivelevich, and Tyomkyn, though they wrote ``needless to say that we expect these questions to be difficult''. To be precise, for an $r$-uniform hypergraph $H$ and for $0\le \ell\le \binom kr$, we define $N(n,H)$ and $N_r(k,\ell)$ to be the maximum possible numbers of $k$-vertex subsets in an $n$-vertex $r$-uniform hypergraph which induce a copy of $H$ and which induce exactly $\ell$ edges, respectively. Then we can define $\operatorname{ind}(H)$ and $\operatorname{ind}_r(k,\ell)$ as in \cref{eq:ind(H),eq:ind(k ell)}, and use these notions to generalise the large inducibility and edge-statistics conjectures in the obvious ways: namely, the maximum value of $\operatorname{ind}(H)$, over all $k$-vertex $r$-uniform hypergraphs which are neither empty or complete, and the maximum value of $\operatorname{ind}_r(k,\ell)$ over all $0<\ell<\binom kr$, both tend to $1/e$ as $k\to \infty$ (holding $r$ fixed).

In a combination
of papers by Kwan, Sudakov and Tran~\cite{KST19}, Fox and Sauermann~\cite{FS20},
and Martinsson, Mousset, Noever and Truji\'c~\cite{MMNT19}, the edge-statistics
conjecture for graphs (i.e.~$r=2)$, and therefore the large inducibility conjecture for graphs, have
been resolved. These papers also provide evidence for the hypergraph edge-statistics conjecture, establishing it in the special cases where $\ell = o(k)$ and where $r = 3$ and
$\ell = \Omega(k^3)$. Our first main theorem completely resolves the hypergraph edge statistics conjecture, and therefore the hypergraph large inducibility conjecture. 

\begin{theorem}
\label{conj:1/e}Fix any $r\in\mb N$ and $\varepsilon>0$. Suppose $k$ is sufficiently
large in terms of $r,\varepsilon$. If $\ell\notin\{0,\binom{k}{r}\}$, then
\[
\operatorname{ind}_{r}(k,\ell)\le\frac{1}{e}+\varepsilon.
\]
Consequently, for any $k$-vertex $r$-uniform hypergraph $H$ that is neither empty nor complete, 
\[
\operatorname{ind}(H)\le\frac{1}{e}+\varepsilon.
\]
\end{theorem}

Alon, Hefetz, Krivelevich and Tyomkyn were also interested in the value of $\operatorname{ind}_2(k,\ell)$ when $\ell$ is \emph{far} from $0$ and $\binom k2$. They made several conjectures in this direction~\cite[Conjecture~6.1 and~6.2]{AHKT20}, which have since been resolved by Kwan, Sudakov and Tran~\cite{KST19} and Kwan and Sauermann~\cite{KS}. We also prove a theorem along these lines for hypergraphs.
\begin{theorem}\label{conj:dense}
Fix any $r\in\mb N$ and $\varepsilon>0$.
If $\alpha\binom{k}{r}\le\ell\le(1-\alpha)\binom{k}{r}$ for
some $\alpha\in(0,1/2]$,  and if $\alpha k$ is sufficiently large
in terms of $r,\varepsilon$, then
\[
\operatorname{ind}_{r}(k,\ell)\le\frac{1}{(\alpha k)^{1/2-\varepsilon}}.
\]
\end{theorem}

We note that Kwan, Sudakov and Tran~\cite[Theorem~1.3]{KST19} previously proved the $r=3$ case of \cref{conj:dense} in the ``extremely dense'' regime where $\alpha$ is a constant which may not vary with $k$ (using the induced hypergraph removal lemma).

\cref{conj:1/e,conj:dense} are both essentially optimal, as follows. Write $N(G,k,\ell)$ for the number of $k$-vertex subsets with $\ell$ edges, in an $r$-uniform hypergraph $G$.
\begin{itemize}
\item It was observed by Fox and Sauermann~\cite{FS20} that for every $r$
the constant ``$1/e$'' in \cref{conj:1/e} cannot be improved.
To briefly explain why: for any $1\le s\le r$, let $F$ be a random
$s$-uniform hypergraph on $n$ vertices, in which every possible
edge is present with probability $1/\binom{k}{s}$ independently.
Then, let $G$ be the $r$-uniform hypergraph on the same vertex set 
whose edges are the $r$-sets which are supersets of some edge of $F$. For $\ell=\binom{k-s}{r-s}$,
it is easy to establish the convergence in probability
\begin{equation}
\frac{N(G,k,\ell)}{\binom{n}{k}}\overset{\Pr}{\to}\left(1-\frac{1}{\binom{k}{s}}\right)^{\binom{k}{s}-1}>\frac{1}{e}\label{eq:poisson-computation}
\end{equation}
as $n\to\infty$, which implies that $\operatorname{ind}_{r}(k,\ell)>1/e$.
\item It is easy to see that the exponent ``$1/2$'' in \cref{conj:dense}
cannot be improved. Indeed, consider any $\alpha\in(0,1/2)$ and $k\in\mb N$
such that $\alpha k$ is an integer, and consider any $n$ divisible
by $k$. Let $G$ be an $n$-vertex $r$-uniform hypergraph with a distinguished
set $S$ of $\alpha n$ vertices, whose edges are the $r$-sets that
intersect $S$ in exactly one vertex. Then, with $\ell=\alpha k\binom{k-\alpha k}{r-1}$
(so $\ell$ has order of magnitude $\alpha\binom{k}{r}$), one can
compute
\begin{equation}
\lim_{n\to\infty}\frac{N(G,k,\ell)}{\binom{n}{k}}=\binom{k}{\alpha k}\alpha^{\alpha k}(1-\alpha)^{(1-\alpha)k}\ge\frac{c}{(\alpha k)^{1/2}}\label{eq:gaussian-computation}
\end{equation}
for some absolute constant $c>0$, which implies that $\operatorname{ind}_{r}(k,\ell)>c(\alpha k)^{-1/2}$.
\end{itemize}
The estimates in \cref{eq:poisson-computation} and \cref{eq:gaussian-computation}
can be confirmed via direct computation, but from a more conceptual
point of view, they can also be intuitively understood in terms of
two different approximations for the binomial distribution, which
lead to two different \emph{anticoncentration} bounds (i.e., upper
bounds on point probabilities). On the one hand, a binomial distribution
with a low success probability can be closely approximated by a \emph{Poisson}
distribution. The bound in \cref{eq:poisson-computation} is related
to the fact that a Poisson random variable with parameter 1 is equal
to 1 with probability $1/e$, and this is the maximum possible probability
for any Poisson random variable to take any particular nonzero value.
On the other hand, a binomial distribution with $k$ trials and success
probability $\alpha\in(0,1/2)$ can be approximated by a \emph{Gaussian}
distribution with standard deviation about $\sqrt{\alpha k}$. The
point probabilities of such a binomial distribution are at most about
$1/\sqrt{\alpha k}$; this corresponds to the fact that the corresponding
Gaussian distribution has density at most about $1/\sqrt{\alpha k}$.

The proofs of \cref{conj:1/e,conj:dense} both crucially depend
on anticoncentration inequalities that vastly generalise the above
two observations about binomial distributions. Specifically, we need
two general anticoncentration inequalities for low-degree polynomials
of independent random variables: (a strengthened form of) a ``Poisson-type''
anticoncentration inequality due to Fox, Kwan and Sauermann~\cite{FKS21},
and bounds on the so-called \emph{polynomial Littlewood--Offord problem}
due to Meka, Nguyen and Vu~\cite{MNV16}.

Recall that $\on{ind}_r(k,\ell)$ can be defined in terms of quantities of the form $N(G,k,\ell)$, which can be understood as the probability that a random $k$-vertex subset of $G$ has exactly $\ell$ edges. It is not hard to interpret the number of edges in a random $k$-vertex subset of $G$ as a polynomial evaluated at a random vector (namely, at a random point on a ``slice of the Boolean hypercube''). 
Unfortunately, since the entries of this random vector are not independent, one cannot directly apply the aforementioned polynomial anticoncentration inequalities (in fact, it is easy to see that the conclusions of these inequalities are in general false on a slice of the Boolean hypercube).
The key new contributions in this paper are several different ways to ``transfer'' polynomial anticoncentration inequalities to slices of the Boolean hypercube. For the proof of \cref{conj:dense}, we generalise a coupling lemma due to Kwan, Sudakov and Tran~\cite{KST19}, and combine it with a result of Bollob\'as and Scott~\cite{BS15} related to discrepancy of hypergraphs. For the proof of \cref{conj:1/e}, we additionally use a classical estimate on hypergeometric distributions due to Ehm~\cite{Ehm91} which allows us to transfer Poisson-type anticoncentration inequalities to functions of the slice which only depend on a few vertices. Then, we prove a delicate combinatorial lemma which says that every hypergraph has a small vertex subset $Y$ which ``only sees large matchings''. We use our Poisson-type anticoncentration inequality to understand ``what happens inside $Y$'', and then after conditioning on this information we apply polynomial Littlewood--Offord bounds (which are effective for hypergraphs with large matchings).

\subsection{Further directions}

There is nearly unlimited potential to ask more precise questions
about the quantities $\operatorname{ind}_{r}(k,\ell)$. Most obviously,
there is the question of removing the ``$\varepsilon$'' in the
exponent of \cref{conj:dense}: in the setting of \cref{conj:dense}
we conjecture that 
\begin{equation}
\operatorname{ind}_{r}(k,\ell)\le\frac{C_{r}}{(\alpha k)^{1/2}}\label{eq:stronger-dense}
\end{equation}
for some constant $C_{r}$ depending only on $r$ (this also appears
as \cite[Conjecture~5]{KST19}). In the case $r=2$, this bound was recently
proved by Kwan and Sauermann~\cite{KS},
via new
progress on the so-called \emph{quadratic Littlewood--Offord problem}.
Actually, our proof of \cref{conj:dense} reduces the general-$r$
case of \cref{eq:stronger-dense} to a well-known conjecture in Littlewood--Offord
theory. We introduce the polynomial Littlewood--Offord problem properly,
and discuss these aspects further, in \cref{sec:poly-LO}.

One could also ask about the ``$\varepsilon$'' in \cref{conj:1/e}:
for each $r,k$, what is the exact maximum value of $\operatorname{ind}_{r}(k,\ell)$,
among all $\ell\notin\{0,\binom{k}{r}\}$? We wonder if the maximum is always attained at $\ell=1$. We remark that the exact value of $\operatorname{ind}_{2}(k,1)$, for all $k$, was recently found by Liu, Mubayi and Reiher~\cite[Theorem~1.13]{LMR23} (see also \cite{Hir14} for earlier work in the case $k=4$).

It is also interesting to consider the maximum possible value
of $\operatorname{ind}_{r}(k,\ell)$ among all \emph{pairs} $(k,\ell)$
satisfying $\ell\notin\{0,\binom{k}{r}\}$. In the case $r=2$, it
was suggested by Alon, Hefetz, Krivelevich and Tyomkyn~\cite{AHKT20} that this
maximum value might be $\on{ind}_2(3,1)=3/4$.

Finally, it would be very interesting to investigate ``stability'' in the settings of \cref{conj:1/e,conj:dense}. Although the constant ``$1/e$'' in \cref{conj:1/e} is best-possible in general, we conjecture that it can be improved when $\min(\ell,\binom{k}{r}-\ell)$
is not of the form $\binom{k-s}{r-s}$ (a related theorem was very recently proved by Ueltzen~\cite{Uel} in the setting of graph inducibility). Similarly, although the exponent ``$1/2$'' in \cref{conj:dense} is best-possible in general, we conjecture that there is $\delta_r>0$ such that $\on{ind}_r(k,\ell)\le k^{-1/2-\delta_r}$ for ``generic'' $\ell$ (i.e., for a $1-o(1)$ fraction
of $\ell$ in the range $0\le \ell\le \binom kr$, where asymptotics are as $k\to \infty$, holding $r$ fixed). This seems to be related to a conjecture of Costello~\cite[Conjecture~3]{Cos13} on ``stability'' for the polynomial Littlewood--Offord problem.

\subsection{Notation}We use standard graph theory notation throughout. For a hypergraph $G$, we write $e(G)$ for the number of edges in $G$, and for a vertex subset $U$, we write $G[U]$ to denote the subgraph of $G$ induced by $U$.

We also use asymptotic notation throughout. For functions $f=f(n)$ and $g=g(n)$, we write $f=O(g)$ or $f\lesssim g$ to mean that there is a constant $C$ such that $|f|\le C|g|$, $f=\Omega(g)$ or $f\gtrsim g$ to mean that there is a constant $c>0$ such that $f(n)\ge c|g(n)|$ for sufficiently large $n$, and $f=o(g)$ to mean that $f/g\to0$ as $n\to\infty$. Subscripts on asymptotic notation indicate quantities that should be treated as constants.

For parameters $\alpha, \beta_1,\dots,\beta_q$, we write $\alpha \ll \beta_1,\dots,\beta_q$ to mean ``$\alpha$ is sufficiently small in terms of $\beta_1,\dots,\beta_q$'' (i.e., it is shorthand for a statement of the form ``$\alpha\le f(\beta_1,\dots,\beta_q)$'', for some function $f$ which we do not wish to specify explicitly). Similarly, we write $\alpha \gg \beta_1,\dots,\beta_q$ to mean ``$\alpha$ is sufficiently large in terms of $\beta_1,\dots,\beta_q$''.

For a positive integer $n$ and an integer $0 \leq d \leq n$, we write $[n]=\{1,\dots,n\}$ and denote the set of all size-$d$ subsets of $[n]$ by $\binom{[n]}{d}$. For a real number $x$, the floor and ceiling functions are denoted $\lfloor x\rfloor=\max(i\in \mb Z:i\le x)$ and $\lceil x\rceil =\min(i\in\mb Z:i\ge x)$. We will however sometimes omit floor and ceiling symbols and assume large numbers are integers, when divisibility considerations are not important. All logarithms in this paper without an explicit base are to base $e$, and the natural numbers $\mb N$ do not include zero. 

For a vector $\vec{x}\in \mb R^n$, we write $x_1,\dots,x_n$ for its coordinates, and for $W\subseteq [n]$, we write $\vec x^{W}$ to denote the monomial $\prod_{i \in W}x_i$. For a multilinear polynomial $P \in \mb{R}[x_1,\dots, x_n]$, we denote the coefficient of the monomial $\vec x^{W}$ by $\wh{P}(W)$.   

\subsection{Organisation of the paper}
In \cref{sec:poly-LO}, we introduce the polynomial Littlewood--Offord problem (on anticoncentration of polynomials of independent random variables) and describe the best known bounds for this problem. These results will play a crucial role in the proofs of both \cref{conj:1/e,conj:dense}, but in order to actually apply them we need a general coupling lemma for polynomials ``on a slice of the Boolean hypercube'', which we present in \cref{sec:slice-coupling}. In \cref{sec:dense} we prove \cref{conj:dense}, using the above tools and a result of Bollob\'as and Scott~\cite{BS15}.

Then, in the rest of the paper, we focus on \cref{conj:1/e}. First, in \cref{sec:sparse-LO} we apply the above tools in a different way, to prove an anticoncentration bound for ``sparse'' polynomials with ``large matchings''. In \cref{sec:poisson-LO}, we state and prove our Poisson-type anticoncentration inequality, and in \cref{sec:TV-comparison} we show how to use an estimate of Ehm~\cite{Ehm91} to compare certain functions ``on a slice of the Boolean hypercube'' with corresponding functions of product distributions. Then, in \cref{sec:vertex-cover} we prove a lemma showing that every hypergraph has a small set of vertices which ``only sees large matchings''. Finally, after some technical variance estimates in \cref{sec:variance}, we prove \cref{conj:1/e} in \cref{sec:completing}.

\subsection*{Acknowledgments}We would like to thank Lisa Sauermann for helpful comments. We would also like to thank Alex Grebennikov for identifying an oversight in the application of \cref{thm:slice-vs-product} (in a previous version of this paper).

\section{The polynomial Littlewood--Offord problem}\label{sec:poly-LO}
In this section we introduce the \emph{polynomial
Littlewood--Offord problem}, concerning \emph{anticoncentration} of
polynomials of independent random variables. More specifically, let $P\in\mb R[x_{1},\dots,x_{k}]$ be a $k$-variable polynomial and let $\xi_{1},\dots,\xi_{k}$ be i.i.d.\ Rademacher random variables 
(i.e., $\Pr[\xi_{i}=-1]=\Pr[\xi_{i}=1]=1/2$).
What upper bounds can be proved on the maximum point probability
\[
\sup_{\ell\in\mb R}\Pr[P(\xi_{1},\dots,\xi_{k})=\ell],
\]
in terms of simple combinatorial information about the polynomial
$P$? The most well-known theorem in this direction is due to Erd\H os~\cite{Erd45}:
improving a theorem of Littlewood and Offord~\cite{LO43}, Erd\H os proved
that if $P$ is a linear form with at least $m$ nonzero coefficients,
then
\[
\Pr[P(\xi_{1},\dots,\xi_{k})=\ell]\le\frac{\binom{m}{\lfloor m/2\rfloor}}{2^{m}}\lesssim\frac{1}{\sqrt{m}}.
\]
By letting $P(x_{1},\dots,x_{k})=x_{1}+\dots+x_{m}$
and $\ell=2\floor{m/2}-m$, we see that this bound is exactly best-possible.

For higher-degree polynomials, one cannot hope for a comparable bound
in terms of the number of nonzero coefficients: indeed, the multilinear polynomial
$P(x_{1},\dots,x_{k})=(x_{1}+x_{2})(x_{3}+\dots+x_{k})$ has $2k-4$
nonzero coefficients, but we have $\Pr[P(\xi_{1},\dots,\xi_{k})=0] \ge \Pr[x_1\ne x_2] = 1/2$.
There are a number of different ways that rule out this kind of degenerate
situation; in this paper, we will parameterise $P$ by the \emph{matching
number} of a certain hypergraph associated with $P$.
\begin{definition}
For a multilinear polynomial $P\in\mb R[x_{1},\dots,x_{k}]$ and $a\ge0$, let $H_{a}^{(d)}(P)$
be the $d$-uniform hypergraph on the vertex set $[k]$ with an edge
$I\in \binom{[k]}d$ whenever the coefficient of the monomial $x^I$ has absolute value strictly greater than $a$.
\end{definition}

\begin{definition}
A \emph{matching}
in a hypergraph $H$ is a collection of edges which are pairwise vertex-disjoint.
Let $\nu(H)$ be the maximum number of edges in a matching in $H$.
\end{definition}

It was first proved by Razborov and Viola~\cite{RV13} (building on work of
Rosi\'nski and Samorodnitsky~\cite{RS96} and Costello, Tao and Vu~\cite{CTV06})
that if $\nu(H_{0}^{(d)}(P))\ge m$ then
\begin{equation}
\sup_{\ell\in\mb R}\Pr[P(\xi_{1},\dots,\xi_{k})=\ell]\lesssim_{d}m^{-c_{d}}\label{eq:RV}
\end{equation}
for some $c_{d}>0$ depending only on $d$. That is to say, if $P$
has many degree-$d$ terms with nonzero coefficients, featuring disjoint
sets of variables, then $P(\xi_{1},\dots,\xi_{k})$ is anticoncentrated.

The bound in \cref{eq:RV} has since been improved, but in general
the best possible bound is still unknown. To ensure that the results in this paper are compatible with potential future improvements, we define a function to
describe ``the best possible bound for the polynomial Littlewood--Offord
problem'', as follows.
\begin{definition}
For $d,m\ge 0$, let 
\[
\operatorname{LO}_{d}(m)=\sup_{P,\ell,k}\Pr[P(\xi_1,\dots,\xi_k)=\ell],
\]
where the supremum ranges over all $\ell\in\mb R$, all $k\in \mb N$ and all multilinear polynomials $P\in \mb R[x_1,\dots,x_k]$
of degree at most $d$ 
with $\nu(H_{0}^{(d)}(P))\ge m$,
and we take $\xi_{1},\dots,\xi_{k}$ to be i.i.d.\ Rademacher random
variables.
\end{definition}

By considering the multilinear polynomial obtained from $(x_1+\dots + x_m)^d$
by substituting $x_i^2=1$ for all $i$,
it is easy to see that  
\[
\operatorname{LO}_{d}(m)\gtrsim\frac{1}{\sqrt{m}}.
\]

It is widely believed\footnote{A conjecture along these lines seems to have been first posed by Nguyen
and Vu (see \cite{MNV16,RV13}).} that the matching upper bound $\operatorname{LO}_{d}(m)\lesssim_{d}1/\sqrt{m}$
should also hold. In the case $d=1$ this is
a classical result of Littlewood--Offord and Erd\H os, and in the case $d=2$ this was recently proved by Kwan and
Sauermann~\cite{KS} (improving intermediate results by Costello~\cite{Cos13}).
For general $d$, the best available bound is due to Meka, Nguyen
and Vu~\cite{MNV16} (via a theorem of Kane~\cite{Kan14}), as follows.
\begin{theorem}
\label{thm:poly-LO} For any $d,m\in\mb N$,
\[
\operatorname{LO}_{d}(m)\lesssim_{d}\frac{(\log m)^{O_{d}(1)}}{\sqrt{m}}.
\]
\end{theorem}
We remark that the ``$\eps$'' in \cref{conj:dense}  is entirely
due to the logarithmic factors in \cref{thm:poly-LO}; if we knew that
$\operatorname{LO}_{d}(m)\lesssim_{d}1/\sqrt{m}$, we would be able
to obtain the optimal result \cref{eq:stronger-dense}. Also, we remark
that while the proof of \cref{thm:poly-LO} is a little involved, Razborov
and Viola's proof of the bound $\operatorname{LO}_{d}(m)\lesssim_{d}m^{-c_{d}}$
is very simple, and this weaker bound is enough for our proof of \cref{conj:1/e}.

We will need a version of \cref{thm:poly-LO} that takes terms of all
degrees into account (not just the degree-$d$ terms), as follows.
\begin{corollary}\label{cor:poly-LO}
Consider a multilinear polynomial $P\in\mb R[x_{1},\dots,x_{k}]$ of degree at most $d\ge 0$, and for each
$f\in\{0,\dots,d\}$ let $b_{f}$ be an upper bound on the absolute values of
all degree-$f$ coefficients. Suppose that for some $f\in \{0,\dots,d\}$ and $t\in \mb N$
we have $\nu(H_{a(f,t)}^{(f)})\ge t$, where $a(f,t)=tb_{f+1}+t^{2}b_{f+2}+\dots+t^{d-f}b_{d}$.
Then 
\[
\sup_{\ell\in\mb R}\Pr[P(\xi_{1},\dots,\xi_{k})=\ell]\lesssim_{d}\sup_{f\le d}\operatorname{LO}_{f}(\Omega_{d}(t)).
\]
\end{corollary}

\begin{proof}
We prove the desired statement by induction on $d$. The case $d=0$
holds vacuously, since $\nu(H_{a}^{(0)}(P))\le1$ for any $P,a$.
Fix $d\ge1$ and suppose that the desired bound is true for smaller
$d$. Fix a degree-$d$ multilinear polynomial $P\in\mb R[x_{1},\dots,x_{k}]$,
and suppose that for some $f\le d$ we have $\nu(H_{a(f,t)}^{(f)}(P))\ge t$.

We split into cases. First, suppose that $\nu(H_{0}^{(d)}(P))\ge t/(2d)$.
In this case we have
\[
\sup_{\ell\in\mb R}\Pr[P(\xi_{1},\dots,\xi_{k})=\ell]\le\operatorname{LO}_{d}(t/(2d)),
\]
as desired.

Otherwise, we have $\nu(H_{0}^{(d)}(P)) < t/(2d)$. This means that
$f\le d-1$. Consider a maximum matching $M$ in $H_{0}^{(d)}(P)$,
and let $I\subseteq[k]$ be the set of vertices in this
matching, so $|I|=d|M|\le t/2$. Let $\vec{\xi}[I]:=(\xi_{i})_{i\in I}$,
and consider any outcome $\vec{\xi}[I]=(\xi_{i})_{i\in I}\in\{-1,1\}^{I}$.
Let $P_{\vec{\xi}[I]}\in\mb R[x_{i}:i\notin I]$ be the multilinear polynomial obtained
from $P(x_{1},\dots,x_{k})$ by substituting $x_{i}=\xi_{i}$ for
all $i\in I$. Since $M$ is a maximum matching, $P_{\vec{\xi}[I]}$
has degree at most $d-1$. Also, note that the multilinear degree-$f$
coefficients in $P_{\vec{\xi}[I]}$ differ from the corresponding
coefficients in $P$ by at most 
\[
\binom{|I|}{1}b_{f+1}+\binom{|I|}{2}b_{f+2}+\dots+\binom{|I|}{d-f}b_{d}\le a(f,t)-a(f,t/2),
\]
so all the edges of the induced subgraph $H_{a(f,t)}^{(f)}(P)[[k]\setminus I]$
are also edges of $H_{a(f,t/2)}^{(f)}(P_{\vec{\xi}[I]})$, and therefore, 
\[
\nu(H_{a(f,t/2)}^{(f)}(P_{\vec{\xi}[I]}))\ge t-|I|\ge t/2.
\]
The induction hypothesis then yields
\[
\sup_{\ell\in\mb R}\Pr\big[P(\xi_{1},\dots,\xi_{k})=\ell\big]\leq \sup_{\vec{\xi}[I]}\sup_{\ell\in\mb R}\Pr\big[P_{\vec{\xi}[I]}(\xi_{i}:i\notin I)=\ell\,\big|\,\vec{\xi}[I]\big]\lesssim_{d}\sup_{f\le d-1}\operatorname{LO}_{f}(\Omega_{d}(t)),
\]
as desired.
\end{proof}

\section{Describing the slice via a product measure}\label{sec:slice-coupling}
In this section we record a general lemma describing polynomials of ``slice'' measures in terms of polynomials of independent Rademacher random variables. This is necessary to apply the results from \cref{sec:poly-LO} to study the parameters $\on{ind}_r(k,\ell)$, which can naturally be understood as being about polynomials evaluated on a slice measure.

Specifically, for any hypergraph $G$ on the vertex set $[n]$, and any $k$-vertex subset $U$, we can express
\[
e(G[U]) = \sum_{W \in E(G)} \vec \sigma^W,
\]
where $E(G)$ is the set of edges of $G$, and $\vec{\sigma}$ is the vector in $\{0,1\}^n$ with exactly $k$ ones, defined by setting $\sigma_i = 1$ if $i \in U$ and $\sigma_i = 0$ otherwise. When $U$ is a uniformly random $k$-vertex subset, the corresponding $\vec{\sigma}$ is a uniformly random vector in $\{0,1\}^n$ with exactly $k$ ones (this is called a ``slice of the Boolean hypercube'').

\begin{definition}\label{def:lagrange}
    For an $r$-uniform hypergraph $G$ on the vertex set $[n]$, let $\lambda_G\in\mb R[x_{1},\dots,x_{n}]$ 
    denote the polynomial $\sum_{W\in E(G)}\vec x^W$. In other words, the coefficient $\wh{\lambda_G}(W)$ of $\vec x^W$ is $1$ if $W$ is an edge of $G$, and $0$ if $W$ is not an edge of $G$.
\end{definition}

\begin{definition}\label{def:slice}
    $\on{Slice}(n,k)$ denotes the subset of vectors in $\{0,1\}^n$ with exactly $k$ different 1-entries. We write $\vec \sigma \sim \on{Slice}(n,k)$ to mean that the random vector $\vec \sigma$ is uniformly distributed on $\on{Slice}(n,k)$. 
\end{definition}
Now, note that we can obtain $\on{Slice}(n,k)$ (for $n\ge 2k$) by first randomly choosing $k$ \emph{pairs} of disjoint vertices, then flipping an unbiased coin for each pair to decide which of the pair to actually take (as a 1-entry). We introduce notation for this, as follows.
\begin{definition}\label{def:slice-coupling}
Consider $k,n\in\mb N$. Let \[\vec{v}=\big(v_{1}(-1),v_{1}(1),\dots,v_{k}(-1),v_{k}(1)\big)\]
be a uniformly random sequence of $2k$ distinct elements of $[n]$,
and let $V_{\vec{v}}$ be the set of elements in $\vec v$. Independently from $\vec v$, let $\vec{\xi}=(\xi_{1},\dots,\xi_{k})$
be a sequence of $k$ i.i.d.\ Rademacher random variables. Let $U_{\vec{v},\vec{\xi}}=\{v_{1}(\xi_{1}),\dots,v_{k}(\xi_{k})\}$,
so $U_{\vec{v},\vec{\xi}}$ is a uniformly random subset of $k$ elements
of $[n]$. Define $\vec{\sigma}_{\vec v,\vec \xi}=(\sigma_{1},\dots,\sigma_{n})\in\{0,1\}^n$
by
\[
\sigma_{v}=\begin{cases}
1 & \text{if }v\in U_{\vec v,\vec \xi},\\
0 & \text{otherwise},
\end{cases}
\]
so $\vec{\sigma}_{\vec v,\vec \xi}\sim\operatorname{Slice}(n,k)$.
\end{definition}
The setup in \cref{def:slice-coupling} allows us to interpret a polynomial on the slice as a polynomial in Rademacher random variables, as follows (a similar slightly weaker result appears as \cite[Lemma~2.8]{KST19}).
\begin{lemma}\label{lem:poly-coupling}
Recall the setup in \cref{def:slice-coupling}, and consider an $n$-variable
multilinear polynomial $\lambda = \sum_{W\subseteq [n]}\wh{\lambda}(W)\vec x^{W}\in\mb R[x_{1},\dots,x_{n}]$. 
Then, $\lambda(\vec{\sigma}_{\vec v,\vec \xi})$ is a multilinear polynomial 
of $\vec{\xi}$, with coefficients that depend on $\vec{v}$. Specifically, we have
\[
\lambda(\vec{\sigma}_{\vec v,\vec \xi})=\sum_{I\subseteq[k]}A_{\vec{v}}(I)\vec{\xi}^{\,I},
\]
for coefficients $A_{\vec{v}}(I)$ given explicitly by
\[
A_{\vec{v}}(I)=\sum_{W \in \mc{W}_{\vec{v}}(I)}(-1)^{|W\cap\{v_{i}(-1):i\in I\}|}\wh{\lambda}(W)2^{-|W|},
\]
where $\mc{W}_{\vec{v}}(I)$ is the collection of all subsets $W\subseteq V_{\vec v}$ satisfying $|W\cap\{v_{i}(-1),v_{i}(1)\}|=1$ for every $i\in I$. 
\end{lemma}
\begin{remark}
\label{rmk:coeff-bound}
In particular, if $\lambda$ has degree at most $d$ and all coefficients of $\lambda$ have absolute value at most $q$, then for any $I\subseteq [k]$, we have 
\[|A_{\vec{v}}(I)| \leq q\cdot |\mc{W}_{\vec{v}}(I) \cap \{W\subseteq [n]: |W|\leq d\}|
    \leq q2^{|I|}n^{d-|I|}.\] 
Additionally, if $|I| > d$, then $\mc{W}_{\vec{v}}(I) \cap \{W\subseteq [n]: |W| \leq d\} = \emptyset$, so that $A_{\vec{v}}(I) = 0$. 
\end{remark}

\begin{proof}
For convenience of notation, write $\vec{\sigma} = \vec{\sigma}_{\vec{v}, \vec{\xi}}$. Then, $\sigma_{v_i(1)} = (1+\xi_i)/2$ and $\sigma_{v_i(-1)} = (1-\xi_i)/2$ for all $i \in [k]$, and $\sigma_{w} = 0$ for all other $w$. Let $\mc{W}_{\vec{v}}$ denote the set of all subsets $W\subseteq V_{\vec{v}}$ which satisfy $|W \cap \{v_i(-1), v_i(1)\}| \leq 1$ for each $i$. Note that if $W \notin \mc{W}_{\vec{v}}$, then $\vec \sigma^{W} = 0$. Therefore, 
\begin{align*}
    \lambda(\vec{\sigma}) &= \sum_{W \subseteq [n]} \wh{\lambda}(W) \vec \sigma^{W} = \sum_{W \in \mc{W}_{\vec{v}}}\wh{\lambda}(W)\vec \sigma^W\\
    &= \sum_{W \in \mc{W}_{\vec{v}}}\wh{\lambda}(W)2^{-|W|} \prod_{i: v_i(1) \in W}(1+\xi_i) \prod_{j: v_j(-1) \in W}(1-\xi_j)\\
    &= \sum_{I \subseteq [k]} \left(\sum_{W \in \mc{W}_{\vec{v}}(I)}(-1)^{|W \cap \{v_i(-1): i \in I\}|}\wh{\lambda}(W)2^{-|W|}\right)\vec{\xi}^{\,I} = \sum_{I\subseteq [k]}A_{\vec{v}}(I) \vec{\xi}^{\,I}. \qedhere
\end{align*}
\end{proof}

\section{A strong bound for edge-statistics ``in the bulk''}\label{sec:dense}
In this section we prove \cref{conj:dense}. It will be a simple consequence of the following anticoncentration inequality for edge-statistics of dense hypergraphs.
\begin{lemma}\label{conj:dense-LO}
Let $n=2k$, let $\beta \in (0,1/2]$ and let $G$ be an $n$-vertex $r$-uniform hypergraph with $\beta \binom n r\le e(G)\le (1-\beta)\binom nr$. Then, for $U$ a uniformly random subset of $k$ vertices of $G$, we have 
\[
\sup_{\ell\in\mb R}\Pr[e(G[U])=\ell]\lesssim_r\sup_{f\le r}\operatorname{LO}_{f}(\Omega_r(\beta n)).
\]
\end{lemma}
Before proving \cref{conj:dense-LO}, we see how to deduce \cref{conj:dense}. 
\begin{proof}[Proof of \cref{conj:dense} given \cref{conj:dense-LO}]
 Recall that we are to prove that $\operatorname{ind}_{r}(k,\ell)\le (\alpha k)^{-1/2+\varepsilon}$, where $\operatorname{ind}_{r}(k,\ell)=\lim_{n \to \infty} N_r(n, k, \ell)/\binom nk$ and $N_r(n, k, \ell)$ is the maximum possible number of $k$-vertex subsets of an
$n$-vertex graph which induce exactly $\ell$ edges.  Let $n=2k$, let $G$ be an $n$-vertex graph and let $U$ be a uniformly random subset of $k$ vertices of $G$. Due to the fact that $N_r(n,k,\ell)/\binom nk$ is nonincreasing in $n$, it suffices to prove that 
\[\Pr[e(G[U])=\ell]\le \frac1{(\alpha k)^{1/2-\varepsilon}}.\]
when $\alpha k$ is sufficiently large in terms of $\varepsilon,r$. 

We may assume that $G$ has at least $\alpha \binom k r$ edges and at least $\alpha \binom k r$ non-edges; otherwise it is impossible to have $e(G[U])=\ell$. Recalling that $n=2k$ (so $\binom k r\gtrsim \binom n r$), this means that the assumption in \cref{conj:dense-LO} is satisfied for some $\beta\gtrsim\alpha$, and \cref{conj:dense-LO} implies that 
\[\Pr[e(G[U])=\ell]\lesssim_r \sup_{f\le r}\operatorname{LO}_{f}(\Omega_r(\alpha k)).\]
The desired result then follows from \cref{thm:poly-LO}.
\end{proof}

We need a few key ingredients to prove \cref{conj:dense-LO}. First, we need a lemma essentially due to Bollob\'as and Scott~\cite{BS15}.

\begin{definition}
Consider an $r$-uniform hypergraph $G$ on the vertex set $[n]$. For $W \in \binom{[n]}{r}$, let $\wh G(W)=1$ when $W$ is an edge and let $\wh G(W)=0$ when $W$ is not an edge. 
For $s\in [r]$ define
\[
Q_{s}(G)=\sum_{\vec{x}}\left|\sum_{W}(-1)^{|W\cap\{x_{1}(-1),\dots,x_{s}(-1)\}|}\wh{G}(W)\right|
\]
where the first sum is over sequences 
$\vec{x}=(x_{1}(-1),x_{1}(1),\dots, x_s(-1),x_s(1))$
of $2s$ distinct vertices, and the second sum is over all $W \in \binom{[n]}{r}$ which satisfy $|W\cap\{x_{i}(-1),x_{i}(1)\}|=1$ for
every $i\in [s]$.
\end{definition}

\begin{lemma}\label{lem:BS}
Consider $r,n$ satisfying $n\ge3r$, and consider an $r$-uniform hypergraph $G$ on the vertex set $[n]$, satisfying $\beta\binom nr\le e(G)\le (1-\beta)\binom n r$.
Then, for some $s\in [r]$, 
we have
\[
Q_{s}(G)\gtrsim_{r}\beta n^{r+s}.
\]
\end{lemma}
\begin{proof}
Let $p:=e(G)/\binom{n}{r}$ denote the density of $G$. Since $\beta\binom nr\le e(G)\le (1-\beta)\binom n r$, we have $\beta \le p \le 1-\beta$. Define the function $f:\binom{[n]}r\to \mb R$
by
\[f(W)=\begin{cases}
    1-p &\text{ if }W\text{ is an edge of }G,\\
    -p &\text{ if }W\text{ is not an edge of }G.
\end{cases}\]
Then, we have
\[\|f\|_1=\left(e(G)\cdot (1-p)+\left(\binom nr-e(G)\right)\cdot p\right)=2p(1-p)\cdot \binom nr\ge 2\beta(1-\beta) \cdot \binom nr\gtrsim \beta n^r.\]
Let $(n)_{2s}$ denote the falling factorial $n(n-1)\dots(n-2s+1)$ (i.e., the number of ways to choose a
sequence of $2s$ different vertices); in \cite{BS15}, Bollob\'as and Scott define\footnote{Actually, they first define the $W$-vector in a slightly different way, and prove that the two definitions are equivalent in \cite[Lemma~8]{BS15}.} the ``$W$-vector'' of $f$ to be the vector $(q_0(f),\dots,q_r(f))$ given by
\[q_s(f)=\frac1{(n)_{2s}}\binom{n-2s}{r-s}^{-1}\sum_{\vec{x}}\left|\sum_{W}(-1)^{|W\cap\{x_{1}(-1),\dots,x_{s}(-1)\}|}f(W)\right|
\]
where the first sum is over sequences
$\vec{x}=(x_{1}(-1),x_{1}(1),\dots, x_s(-1),x_{s}(1))$
of $2s$ distinct vertices, and the second sum is over all $W \in \binom{[n]}{r}$ which satisfy $|W\cap\{x_{i}(-1),x_{i}(1)\}|=1$ for
every $i\in [s]$.
Note that $q_s(f)$ is not affected by adding a constant function to $f$, so\[q_s(f)=\frac1{(n)_{2s}}\binom{n-2s}{r-s}^{-1}Q_s(G)\lesssim_r n^{-r-s}Q_s(G)\]  and in particular $q_0(f)=0$.

Then, \cite[Lemma~9]{BS15} says that
\[q_0(f)+\dots+q_r(f)\gtrsim_r n^{-r}\|f\|_1;\]the desired result follows. 
\end{proof}

Next, the following technical lemma shows that random variables of a certain type are unlikely to be very small (this lemma is stated in slightly more general form than we need, as it will also be applied again later in the paper). It is proved by a simple application of Chebyshev's inequality. Given $f\in \mb N$ and a set $V$, an \emph{ordered $f$-subset} of $V$ 
is a sequence of $f$ distinct elements of $V$. We say that two ordered $f$-subsets $\vec X,\vec Y$ of $V$ are \emph{disjoint} if there is no element of $V$ that appears in both $\vec X$ and $\vec Y$.
\begin{lemma}\label{lem:concentration-sequences}
Consider $f,m\in\mb N$ and a set $V$, with $2f\le mf\le |V|$.
For each $i\in [m]$, let $\mathcal{F}_{i}$ be a collection
of at least $\gamma |V|^{f}$ ordered $f$-subsets of $V$, and let $\vec X_{1},\dots,\vec X_{m}$
be uniformly random pairwise disjoint
ordered $f$-subsets of $V$. Let $N$ be the number
of $i$ such that $\vec X_{i}\in\mathcal{F}_{i}$. Then
\[
\Pr[N<\gamma m/2]\lesssim_{f}\frac{1}{\gamma m}.
\]
\end{lemma}

\begin{proof}
Let $\one_{i}$ be the indicator random variable for the event that
$\vec X_{i}\in\mathcal{F}_{i}$, and write $n=|V|$. By linearity of expectation, 
$$ \mb E[N]= \sum_{i=1}^m \mb E[\one_{i}]=
 \sum_{i=1}^m\frac{|\mathcal{F}_{i}|}{n(n-1)\dots(n-f+1)}\ge
\sum_{i=1}^m \frac{|\mathcal{F}_{i}|}{n^{f}}\ge\gamma m.$$
Then, for all $i,j \in [m]$ with $i\neq j$, we have
\begin{align*}
\operatorname{Cov}[\one_{i},\one_{j}] & \le\frac{|\mathcal{F}_{i}|\cdot|\mathcal{F}_{j}|}{n(n-1)\dots(n-2f+1)}-\left(\frac{|\mathcal{F}_{i}|}{n(n-1)\dots(n-f+1)}\right)\left(\frac{|\mathcal{F}_{j}|}{n(n-1)\dots(n-f+1)}\right)\\
 & \lesssim_{f}\frac{|\mathcal{F}_{i}|}{n^{f}}\cdot\frac{|\mathcal{F}_{j}|}{n^{f}}\cdot\frac{1}{n},\\
\operatorname{Var}[N] & \le\sum_{i}\mb E[\one_{i}]+\sum_{i\ne j}\operatorname{Cov}[\one_{i},\one_{j}]\\
 & \lesssim_{f}\sum_{i}\frac{|\mathcal{F}_{i}|}{n^{f}}+\frac{1}{n}\cdot\left(\sum_{i}\frac{|\mathcal{F}_{i}|}{n^{f}}\right)^{2}\le\mb E[N]+\frac{\mb E[N]^{2}}{n}.
\end{align*}
By Chebyshev's inequality, we deduce the desired result
\[
\Pr[N<\gamma m/2]\le\Pr\big[N<\mb E[N]/2\big]\le\frac{\operatorname{Var}[N]}{(\mb E[N]/2)^{2}}\lesssim_{f}\frac{1}{\mb E[N]}+\frac{1}{n}\lesssim\frac{1}{\gamma m}.\qedhere
\]
\end{proof}

Given the conclusion of \cref{lem:BS}, we can apply \cref{lem:concentration-sequences} in the setting of \cref{lem:poly-coupling}, as follows.

\begin{lemma}\label{lem:BS-subsampling}
Consider $r,k\in\mb N$ satisfying $k\ge r$, and let $n=2k$. Consider an $r$-uniform hypergraph $G$ on the vertex set $[n]$, and its corresponding polynomial $\lambda_G$. Suppose
$Q_{s}(G)\ge\beta n^{r+s}$ for some $\beta>0$ and $s\in [r]$. Recall the coefficients $A_{\vec{v}}(I)$ from \cref{lem:poly-coupling} (defined in terms of a random sequence $\vec v$). 

Then, except with probability at most $O_r(1/(\beta n))$ over the randomness of $\vec{v}$, the following holds: there exist $t \gtrsim_r \beta n$ disjoint $s$-sets $I_1,\dots, I_t \subseteq [k]$ such that $|A_{\vec{v}}(I_j)|\gtrsim_{r}\beta n^{r-s}$ for each $j$.
\end{lemma}

\begin{proof}
For a sequence
$\vec{x}=(x_{1}(-1),x_{1}(1),\dots, x_{s}(-1),x_{s}(1))$
of $2s$ distinct indices, let 
\[
a(\vec{x})=\left|\sum_{W}(-1)^{|W\cap\{x_{1}(-1),\dots,x_{s}(-1)\}}\wh{G}(W)\right|
\]
where the sum is over all $W \in \binom{[n]}{r}$ which satisfy
$|W\cap\{x_{i}(-1),x_{i}(1)\}|=1$ for every $i\in [s]$. By definition, we have
$Q_{s}(G)=\sum_{\vec{x}}a(\vec{x})$.
Note that we always have $a(\vec{x})\le2^{s}n^{r-s}$, so 
it must be the case that $a(\vec{x})\gtrsim_{s}\beta n^{r-s}$
for at least $\Omega_{s}(\beta n^{2s})$ different $\vec{x}$. Denote the set of all such $\vec{x}$ by $\mc{F}$.

Now, let $m = \floor{k/s}\gtrsim_s n$. Recall the random sequence $\vec v=(v_{1}(-1),v_{1}(1),\dots,v_{k}(-1),v_{k}(1))$ from \cref{def:slice-coupling}, and for each $j\in [m]$, let
\begin{align*}I_j&=\{s(j-1)+1,s(j-1)+2,\dots,sj\},\\ \vec X_j&=\big(v_{s(j-1)+1}(-1),v_{s(j-1)+1}(1),\dots,v_{sj}(-1),v_{sj}(1)\big).\end{align*}
So, in the language of \cref{lem:concentration-sequences}, $\vec X_1,\dots,\vec X_m$ are uniformly random disjoint ordered $2s$-subsets of $[n]$. If $\vec X_j\in \mc F$, then
$|A_{\vec v}(I_j)|=2^{-r}a(\vec{X}_j)\gtrsim_r \beta n^{r-s}$. So, the desired result follows from \cref{lem:concentration-sequences}, with $f=2s$ and 
$\gamma\gtrsim_r \beta$ and $\mathcal F_j=\mathcal F$ for all $j\in [m]$.
\end{proof}
Now, it is straightforward to deduce \cref{conj:dense-LO}.

\begin{proof}[Proof of \cref{conj:dense-LO}]
By \cref{lem:BS} and \cref{lem:BS-subsampling}, there is $s\in [r]$ and $t \gtrsim_r \beta n$ such that, except with probability $O_r(1/(\beta n))\lesssim \sup_{f\le d}\on{LO}_f(\beta n)$ over the randomness of $\vec{v}$, the following holds: there are at least $t$ disjoint $s$-sets $I_1,\dots, I_t \subseteq [k]$ such that $|A_{\vec{v}}(I_j)|> r2^r t n^{r-s-1}$ for each $j$. Condition on such an outcome of $\vec v$.

Recalling \cref{rmk:coeff-bound}, we have $|A_{\vec v}(I)|\le 2^r n^{r-f}=:b_f$ for all $I \in \binom{[k]}{f}$. So, in the notation of \cref{cor:poly-LO}, we have \[|A_{\vec{v}}(I_j)|>r2^r t n^{r-s-1}\ge a(s,t)\] for each $j \in [t]$. We can therefore apply \cref{cor:poly-LO} using the randomness of $\vec \xi$ to obtain the desired result.
\end{proof}

\section{A Littlewood--Offord-type inequality for ``sparse'' polynomials on the slice}\label{sec:sparse-LO}
For the proof of \cref{conj:1/e} we will need one further application of polynomial Littlewood--Offord bounds. Here we consider a degree-$d$ polynomial $\lambda$ which is ``sparse'' (almost all its top-degree coefficients are zero), and give a bound in terms of the matching number of the hypergraph of nonzero degree-$d$ coefficients of $\lambda$. This matching number is \emph{a priori} quite different to the matching number of the polynomial obtained from \cref{lem:poly-coupling}; we take advantage of our sparseness assumption to relate the two (this is actually a somewhat delicate matter, and requires a careful minimality argument).

Recall that the notation $\alpha \gg \beta_1,\dots,\beta_q$ means ``$\alpha$ is sufficiently large in terms of $\beta_1,\dots,\beta_q$''.
\begin{lemma}\label{lem:sparse-LO}
For $d,q\in\mb N$, there is $\delta>0$ such that the following holds.
Consider $R,k,m,n\in\mb N$ with $2k\le n\le Rk$ and let $\lambda\in\mb R[x_{1},\dots,x_{n}]$
be an $n$-variable multilinear polynomial with degree at most $d$, whose coefficients
all lie in $\{0,\dots,q\}$. Suppose $\nu(H^{(d)}_0(\lambda))\ge m$ and suppose
that $\lambda$ has at most $\delta n^{d}$ nonzero degree-$d$ terms.
Then, for $\vec{\sigma}\sim\operatorname{Slice}(n,k)$, provided that $k \gg d,q,R$, we have 
\[
\sup_{\ell\in\mb R}\Pr[\lambda(\vec{\sigma})=\ell]\lesssim_{d,R}\sup_{f\le d}\operatorname{LO}_{f}(\Omega_{R,d,q}(m))\lesssim_{d,R,q} \frac 1{m^{1/3}}.
\]
\end{lemma}
For our proof of \cref{lem:sparse-LO} we use the following simple lemma, showing that if a graph $G$ has a large matching, then a random subset of its vertices also typically has a large matching.
\begin{lemma}\label{lem:matching-subsample}
Let $G$ be a $d$-uniform hypergraph whose number of vertices is between $2k$ and $Rk$, and which satisfies $\nu(G)\ge m$. Let $U$ be a uniformly random $k$-vertex subset, for some $k\ge 2d$. Then, except with probability $O_{d,R}(1/m)$, we have $\nu(G[U])\gtrsim_{d,R} m$.
\end{lemma}
\begin{proof}
Let $n$ be the number of vertices in $G$.
First, note that we have the trivial upper bound $m\le n/d\lesssim_R k$. By shrinking $m$ if necessary, note that we can assume $k/m\gg q,\delta,d,R$ (we would only shrink $m$ to a value of the form $\Omega_{R,d,q}(m)$, so the form of the final bound would be unchanged).

Let $E_1,\dots,E_m$ be the edges of a matching in $G$. For each $i\in [m]$, let $\one_i$ be the indicator random variable for the event $E_i\subseteq U$, and let $N=\one_1+\dots+\one_m$. 
Then, with similar calculations as in \cref{lem:concentration-sequences} we have\footnote{Actually, with slightly more careful calculations one can see that $\on{Cov}[ \one_i,\one_j]\le 0$ (this is essentially done in the proof of \cref{lem:cheap-variance}).}
\[\mb E[\one_i]=\frac{\binom {n-d}{k-d}}{\binom {n}k}\gtrsim_{d,R}1,\quad \on{Cov}[ \one_i,\one_j]=\frac{\binom {n-2d}{k-2d}}{\binom {n}k}-\left(\frac{\binom {n-d}{k-d}}{\binom {n}k}\right)^2\lesssim_{d,R} 1/k,\]
and therefore 
\[\mb E[N]\gtrsim_{d,R}m,\quad \on{Var}[N]\lesssim_{d,R}m+m^2/k\lesssim_R m\]
(using that $m\le (Rk)/d\lesssim_R k$). The desired result follow from Chebyshev's inequality.
\end{proof}
Now, we prove \cref{lem:sparse-LO}. 
\begin{proof}[Proof of \cref{lem:sparse-LO}]
Recall the coefficients $A_{\vec{v}}(I)$ from \cref{lem:poly-coupling}, defined in terms of a random sequence $\vec v=(v_{1}(-1),v_{1}(1),\dots,v_{k}(-1),v_{k}(1))$. Let $\vec v(1)=(v_1(1),\dots,v_k(1))$. By \cref{lem:matching-subsample}, over the randomness of $\vec v(1)$, except with probability $O_{d,R}(1/m)\lesssim_{d,R} \on{LO}_d(m)$, the subgraph of $H_0^{(d)}(\lambda)$ induced by the vertices of $\vec v(1)$ has a matching of size $m' = \Omega_{R,d}(m)$. 
Condition on such an outcome of $\vec v(1)$, and let $E_1,\dots,E_{m'}$ be the edges of this matching. Given such a choice of $\vec{v}(1)$, condition on an arbitrary outcome of $V_{\vec v}$ such that $V_{\vec{v}} \supseteq \{v_1(1),\dots, v_k(1)\}$. Note that there is still some randomness remaining in $\vec v$: namely, $v_1(-1),\dots,v_k(-1)$ is a uniformly random ordering of the vertices in $U(-1):=\{v_1(-1),\dots,v_k(-1)\}$.

For a set $W\subseteq V_{\vec v}$, let $B(W)$ be the number of nonzero coefficients
$\wh{\lambda}(Z)$ of $\lambda$, among all size-$d$ sets $Z$ satisfying $W\subseteq Z\subseteq V_{\vec v}$. By assumption, \begin{equation}
B(\emptyset)\le\delta n^{d}. \label{eq:B(0)}
\end{equation}
Consider $M \gg \delta, q,d$. For all $j \in [m']$, we have 
$B(E_{j})=1=(n/M)^{d-|E_{j}|}$, so we may define $F_{j}\subseteq E_{j}$
to be a \emph{minimal} subset of $E_{j}$ satisfying $B(F_{j})\ge(n/M)^{d-|F_{j}|}$.
Recalling \cref{eq:B(0)}, and assuming $\delta < M^{-d}$, we have $F_{j}\ne\emptyset$. Assume without loss of generality that all the sets $F_1,\dots,F_{m'/d}$ have the same size $d' \in [d]$.

Next, for each $j\in [m'/d]$, let $I_j=\{i:v_i(1)\in F_j\}\subseteq[k]$. For each $j\in [m'/d]$, we have   $|I_j|=|F_j|=d'$, since $F_j\subseteq E_j\subseteq \{v_1(1),\ldots,v_k(1)\}$. Define the sequence $\vec X_j=(v_i(-1):i\in I_j)$ and its underlying set $X_j=\{v_i(-1):i\in I_j\}$. 
Then, note that each coefficient $A_{\vec v}(I_j)$ depends only on the randomness of $\vec X_j$. Recall that since we have conditioned on a choice of the set $U(-1)$, the number of possibilities for $\vec{X}_j$ is $k(k-1)\cdots (k-d'+1)\leq k^{d'}$.
\begin{claim*}
Let $M \gg q,d,R$ and $\alpha\ll q,M,R$. For each $j\in[m'/d]$, let $\mc F_j$ be the collection of all possible outcomes of $\vec X_j$ which satisfy $A_{\vec v}(I_j)\ge \alpha n^{d-d'}$. Then $|\mc F_j| \gtrsim k^{d'}$.
\end{claim*}
\begin{claimproof}
Consider an arbitrary $j \in [m'/d]$ and let  $\varepsilon=1/M^{d-d'}$,
so that $B(F_{j})\ge\varepsilon n^{d-d'}.$ By the definition
of $F_{j}$, for every \emph{proper} subset $Y\subsetneq F_{j}$,
we have $B(Y)<(n/M)^{d-|Y|}\le(\varepsilon/M)n^{d-|Y|}$. The sum
of $B(T)$ over all size-$d'$ sets $T\supseteq Y$ satisfying $T\setminus Y\subseteq U(-1)$ is therefore at most 
\[
\binom{d-|Y|}{d'-|Y|}B(Y)\lesssim_{d,R}\frac{\varepsilon}{M} n^{d-d'}\binom{k-|Y|}{d'-|Y|}.
\]
That is to say, if $T$ is \emph{uniformly random} subject to the constraints $T\supseteq Y$ and $T\setminus Y\subseteq U(-1)$, then $\mb E[B(T)]\lesssim_{d,R} (\varepsilon/M) n^{d-d'}$, so by Markov's inequality, provided $M$ is sufficiently large in terms of $d,R$, we have
\begin{equation}
    \Pr\big[B(T)>(\varepsilon/M^{1/3}) n^{d-d'}\big]\le M^{-1/3}.\label{eq:markov}
\end{equation}
Now, say that a set $X\subseteq U(-1)$ is \emph{good} if $B(T)\le (\varepsilon/M^{1/3}) n^{d-d'}$
for every size-$d'$ subset $T\subseteq F_{j}\cup X$ containing
at least one vertex of $X$. 
If $X_j$ is good, then
\begin{align*}
    A_{\vec{v}}(I_j) &\geq 
    \sum_{W\supseteq F_j}\wh{\lambda}(W)2^{-|W|} - 
    \sum_{W\not \supseteq F_j}\wh{\lambda}(W)2^{-|W|}\\
    &\ge \left(B(F_j)-d'\binom{2k-1}{d-d'-1}\right)\cdot 2^{-d} - q\cdot \max_T B(T)\\
    &\gtrsim_d \eps n^{d-d'}-(\eps/M^{1/3}) n^{d-d'}\\
    &\gtrsim_{d,M} n^{d-d'}
\end{align*}
where the sums on the first line are over all $W\subseteq V_{\vec v}$ which
satisfy $|W\cap\{v_{i}(-1),v_{i}(1)\}|=1$ for every $i\in I_j$, and the ``max'' in the second line is over all size-$d'$ subsets $T\subseteq F_{j}\cup X_j$ containing
at least one vertex of $X_j$.

So, it suffices to show that $X_j$ is good with probability $\Omega(1)$ (recall that the randomness comes from a uniformly random permutation of the elements of $U(-1)$). To this end, for each proper subset $J\subsetneq I_j$, let $Y_J=\{v_i(1):i\in J\}\subseteq F_j$ and let $T_J$ be the random size-$d'$ set containing $v_i(1)$ for $i\in J$ and $v_i(-1)$ for $i\notin J$. If $B(T_J)\le (\varepsilon/M^{1/3}) n^{d-d'}$ for all $J\subsetneq I_j$, then $X_j$ is good. But note that each $T_J$ is a uniformly random set satisfying $T_J\supseteq Y$ and $T\setminus Y\subseteq U(-1)$, so by \cref{eq:markov} and a union bound over all $J\subsetneq I_j$, we see that $\Pr[X_j\text{ is not good}]\le 2^d\cdot M^{-1/3}\le 1/2$.
\end{claimproof}
Now, recall from \cref{rmk:coeff-bound} that we always have $|A_{\vec v}(I)|\le q 2^f (2k)^{d-f}=:b_f$ for all $I\in\binom{[k]}{f}$. As in \cref{cor:poly-LO}, let $a(d',t)=t b_{d'+1}+\dots+t^{d-d'}b_d\lesssim_{q,d}mk^{d-d'-1}$.

In the notation of \cref{lem:concentration-sequences}, note that $\vec X_1,\dots,\vec X_{m'/d}$ are uniform random ordered $d'$-subsets of $U(-1)$. 
So, by \cref{lem:concentration-sequences} (with $f=d'$ and $\gamma=\Omega(1)$) and the above claim, except with probability $O(1/m)\lesssim \on{LO}_d(m)$ we have $|A_{\vec v}(I_j)|\gtrsim_{q,M,d,R} k^{d-d'}$, and therefore $|A_{\vec v}(I_j)|>a(d',t)$, for $t=\Omega(m'/d)$ different $j$ (here we are using that $k/m\gg q,\delta,d,R$). Condition on an outcome of $\vec v$ such that this is the case. 

We may then apply \cref{cor:poly-LO}, using the randomness of $\vec \xi$, to obtain the desired conclusion.
\end{proof}

\section{A Poisson-type anticoncentration inequality}\label{sec:poisson-LO}
Another key ingredient for the proof of \cref{conj:1/e} is a ``Poisson-type'' anticoncentration inequality for polynomials of independent random variables, strengthening a result of Fox, Kwan and Sauermann~\cite[Theorem~1.8]{FKS21}. 

\newcommand{\yvar}{s}

\begin{theorem}\label{thm:Poisson-anticoncentration}
Fix $\gamma>0$, and suppose $p$ is sufficiently small in terms of $\gamma$. Let $F\in \mb R[x_1,\dots,x_s]$ be a multilinear polynomial which has nonnegative coefficients, and zero constant coefficient. Let $\beta_1,\dots,\beta_s$ be i.i.d.\ $\on{Bernoulli}(p)$ random variables. Then for any $t\in [0,\infty)$ and $\ell > 3^\yvar t$,
we have
\[\Pr\big[|F(\beta_1,\dots,\beta_\yvar)-\ell|\le t\big]\le \frac1e+\gamma.\]
\end{theorem}
\begin{proof}
 We prove by induction on $\yvar$ that whenever $\ell>3^\yvar t$ and $F\in \mb R[x_1,\dots,x_s]$ is a multilinear polynomial which has nonnegative coefficients and zero constant coefficient, we have
    \[\Pr\big[|F(\vec \beta)-\ell|\le t\big]\le \max_{n\in\mb N} \Pr[X_{n,p}=1],\]
    where $X_{n,p}\sim \on{Bin}(n,p)$. This suffices, because elementary estimates show that the right-hand side of the above expression converges to $1/e$ as $p\to 0$ (see \cite[Lemma~3.3]{FKS21}).
    
    The desired statement is vacuously true for $\yvar=0$, so fix some $\yvar\ge 1$ and assume the desired statement is true for smaller $\yvar$. Write $\|\vec \beta\|_0$ for the number of nonzero entries of $\vec \beta$.

    Let $a_i$ be the coefficient of $x_i$ in $F(x_1,\dots,x_\yvar)$. First, if we have $a_i> (\ell+t)/2$ for all $i$, then we can only have $|F(\vec \beta)-\ell|\le t$ if $\|\vec \beta\|_0=1$.
    Indeed, if we had $\|\vec\beta\|_0=0$ we would have $F(\vec \beta)=0$, and if we had $\|\vec\beta\|_0\ge 2$ we would have $F(\vec \beta)-\ell> 2(\ell+t)/2-\ell=t$. The desired bound follows.

    Otherwise, suppose without loss of generality that $a_\yvar\le (\ell+t)/2$. We can write $F(x_1,\dots,x_\yvar)=G(x_1,\dots,x_{\yvar-1})+a_\yvar x_\yvar+x_\yvar H(x_1,\dots,x_{\yvar-1})$, where $G,H\in \mb R[x_1,\dots,x_{s-1}]$ are multilinear polynomials which have nonnegative coefficients, and zero constant coefficient. We will show the desired bound conditional on both of the two possible outcomes of $\beta_\yvar$.

    Let $\pvec\beta=(\beta_1,\dots,\beta_{\yvar-1})$. If $\beta_\yvar=0$, then $|F(\vec \beta)-\ell|\le t$ if and only if $|G(\pvec \beta)-\ell|\le t$, and the desired result follows by induction. Otherwise, if $\beta_\yvar=1$, then 
    $|F(\beta)-\ell|\le t$ if and only if $|(G+H)(\pvec \beta)-(\ell-a_\yvar)|\le t$. Note that $\ell-a_\yvar\ge (\ell-t)/2\ge \ell/3 > 3^{\yvar-1}t$, so the desired result again follows by induction.
\end{proof}

\section{Comparison between the slice and a product measure}\label{sec:TV-comparison}
Recall that in order to apply bounds on the polynomial Littlewood--Offord problem, we needed a lemma describing a slice measure in terms of a product measure (\cref{lem:poly-coupling}). In order to apply \cref{thm:Poisson-anticoncentration}, we also need a comparison between a slice measure and a product measure, but since \cref{thm:Poisson-anticoncentration} only applies to polynomials with nonnegative coefficients, we will need a lemma of a rather different flavour, for functions on the slice that only depend on a few coordinates.

For probability distributions $\mu,\nu$ with the same sigma-algebra of events, recall that the \emph{total variation distance} $\on{d}_{\mr{TV}}(\mu,\nu)$ is the supremum of $|\mu(A)-\nu(A)|$ over all events $A$.
\begin{theorem}\label{thm:slice-vs-product}
Let $k\le n/2$.
Consider any set $S$ and function $F:\{0,1\}^n\to S$, such that $F(x_1,\dots,x_n)$ only depends on $x_1,\dots,x_{\yvar}$.     Let $\vec \sigma\sim \on{Slice}(n,k)$ and let $\vec \beta$ be a vector of $n$ independent $\on{Bernoulli}(k/n)$ random variables. Then 
    \[\on d_{\mr{TV}}\!\big(F(\vec \sigma),F(\vec \beta)\big)\leq \frac{\max(\yvar,2n/k)-1}{n-1}.\]
\end{theorem}
\newcommand{\newvar}{t}
We will prove \cref{thm:slice-vs-product} using the following theorem of Ehm~\cite{Ehm91}. Write $\on{Hyp}(n,k,\newvar)$ for the hypergeometric distribution with $k$ draws from a population of size $n$ with $\newvar$ ``featured'' elements, and let $\on{Bin}(n,p)$ be the binomial distribution with $n$ trials and success probability $p$.
\begin{theorem}[{\cite[Theorem~2]{Ehm91}}]
\label{thm:Ehm}
If $(k/n)(1-k/n)\newvar\ge 1$ then
    \[\on d_{\mr{TV}}\!\big(\on{Hyp}(n,k,\newvar),\on{Bin}(\newvar,k/n)\big)\leq \frac{\newvar-1}{n-1}.\]
\end{theorem}

\begin{proof}[Proof of \cref{thm:slice-vs-product}]
Let $\newvar=\max(s,2n/k)$, so $(k/n)(1-k/n)\newvar \ge 1$ and $F(x_1,\dots,x_n)$ only depends on $x_1,\dots,x_t$ (since $\newvar\ge \yvar$). For $\vec x\in \{0,1\}^n$, let $\|\vec x\|_{0;\newvar}$ be the number of $i\in [\newvar]$ for which $x_i=1$.  Let $\vec\sigma \sim \on{Slice}(n,k)$ and let $\vec \beta$  be a vector of $n$ independent $\on{Bernoulli}(k/n)$ random variables. Then clearly $\|\vec \sigma\|_{0;\newvar}\sim \on{Hyp}(n,k,\newvar)$ and $\|\vec \beta\|_{0;\newvar}\sim \on{Bin}(\newvar,k/n)$. For all $q$ in the support of both $\on{Hyp}(n,k,\newvar)$ and $\on{Bin}(\newvar,k/n)$, the conditional distribution of $F(\vec \sigma)$ given $\|\vec \sigma\|_{0;\newvar}=q$ is the same as the conditional distribution of $F(\vec \beta)$ given $\|\vec \beta\|_{0;\newvar}=q$. So,
    \[\on d_{\mr{TV}}\!\big(F(\vec{\sigma}), F(\vec{\beta})\big) \leq \on d_{\mr{TV}}\!\big(\|\vec \sigma\|_{0;\newvar}, \|\vec \beta\|_{0;\newvar}\big) \leq \frac{\newvar-1}{n-1},\]
    where the last inequality follows from \cref{thm:Ehm}.
\end{proof}

\section{The vertex cover lemma}\label{sec:vertex-cover}
In this section we prove a lemma showing that every hypergraph $G$ has a small vertex set $Y$ that ``only sees large matchings'', in the sense that for any subset $S\subseteq Y$, if we remove all the edges intersecting $S$, and we remove all the vertices of $Y\setminus S$ from all the remaining edges, the resulting hypergraph either has a large matching or no edges at all.
\begin{definition}
\label{def:G-YX}
Let $G$ be an $r$-uniform hypergraph. For vertex subsets $X\subseteq Y$,
write $G_{Y}(X)=\{e\backslash X:e\in G-(Y\backslash X),\;e\backslash X\ne\emptyset\}$.
In other words, delete all edges which intersect $Y\backslash X$,
and look at the portion of each edge outside $X$. This is a mixed-uniformity
hypergraph, whose edges have sizes between $1$ and $r$. For $d\in[r]$,
let $G_{Y}^{(d)}(X)$ be the subhypergraph of edges which have size
exactly $d$.
\end{definition}

\begin{lemma}\label{lem:vertex-cover}
Let $r, m \in \mb{N}$ and let $G$ be an $r$-uniform hypergraph. We can find a vertex
set $Y$ of size $O_{r,m}(1)$ such that the following holds. Consider
any $X\subseteq Y$ such that $G_{Y}(X)$ is nonempty, and let $d$
be the maximum integer such that $G_{Y}^{(d)}(X)$ is nonempty. Then
$G_{Y}^{(d)}(X)$ has a matching of size at least $m$. Moreover, if $Y\neq \emptyset$, then every edge of $G$ has non-empty intersection with $Y$.
\end{lemma}

\begin{proof}
It will be more convenient to prove the following statement. Let $V$ be the vertex set of $G$. For $S\subseteq Y\subseteq V$, let $\Gamma_{Y}(S):=\{e\backslash S:\,e\in G,\, e\cap Y=S\}$ (this is a hypergraph whose edges have size exactly $r-|S|$). Say
that $S$ is \emph{$Y$-relevant} if $\Gamma_{Y}(S)$ is nonempty and $\Gamma_{Y}(S')$ is empty for all $S'\subsetneq S$. 
We say that a relevant set $S$ is $m$-\emph{bad} (or simply, bad), if the maximum matching in $\Gamma_Y(S)$ has size at most $m -1$. We will show that we can
choose a vertex set $Y$ of size $O_{r,m}(1)$ such that there are no $Y$-relevant sets $S\subseteq Y$ which are bad. To see that this implies the conclusion of the lemma, for $X\subseteq Y$, let $d$ be the maximum integer such that $G_Y^{(d)}(X)$ is nonempty. Since $G_Y^{(d)}(X)$ is nonempty, there exists $S\subseteq X$ of size $r-d$ such that $\Gamma_Y(S)$ is non-empty, and since $d$ is the largest integer satisfying this property, it must be the case that $\Gamma_Y(S') = \emptyset$ for all $S'\subsetneq S$. In other words, $S$ is $Y$-relevant. 
Therefore, $\Gamma_Y(S)$ has a matching of size at least $m$, and since $\Gamma_Y(S)\subseteq G_Y^{(d)}(X)$, our conclusion follows.   

We will construct our desired set $Y$ iteratively, by iterating a ``greedy cover'' map $\phi$ starting from the empty set. Specifically, fix an ordering of the vertex set, and for $Z\subseteq V$, we define $\phi(Z)\supseteq Z$ as follows.

\begin{itemize}
\item If there is some $Z$-relevant set
$S\subseteq Z$ such that $\Gamma_{Z}(S)$ has no matching of size
at least $m$, then consider such a set $S$ with the smallest size (breaking ties lexicographically, according to our ordering of the vertex set) and a maximum matching $M$ in $\Gamma_{Z}(S)$ (again, breaking ties lexicographically), let $W$ be the vertex set of $M$, and set $\phi(Z)=Z\cup W$.
\item Otherwise, set $\phi(Z)=Z$.
\end{itemize}

For $Z\subseteq V$, let $\phi^*(Z)$ be the result of repeatedly applying the map $\phi$, starting with $Z$, until it stabilises, and let $Y=\phi^*(\emptyset)$. We just need to show that $|Y|=O_{r,m}(1)$; this will be a simple inductive consequence
of the following claim.

\begin{claim*}
For $Z\subseteq V$ and $\ell\in\{0,\dots,r\}$, let $N_{\ell}(Z)$ be the
number of $Z$-relevant sets $T\subseteq Z$ with $|T|=\ell$. If $\phi(Z)\ne Z$, then there exists $d \in \{0,\dots, r\}$ for which the following hold: \begin{enumerate}
\item $N_{d}(\phi(Z))<N_{d}(Z)$,
\item $N_{f}(\phi(Z)) \leq N_{f}(Z)$ for all $f<d$,
\item $|\phi(Z)|<|Z|+rm$.
\item $N_{f}(\phi(Z))\le |\phi(Z)|^f< (|Z|+rm)^f$ for all $f>d$,
\end{enumerate}
\end{claim*}
\begin{claimproof}
First, parts (3) and (4) are immediate, because $\phi(Z)$ is obtained by adding fewer than $rm$ vertices to $Z$.

To prove parts (1) and (2), suppose $\phi(Z)\neq Z$ and let $S \subseteq Z$ be the set appearing in the definition of $\phi(Z)$. Note that, by construction, $\Gamma_{\phi(Z)}(S) = \emptyset$. In particular, $S$ is a $Z$-relevant set which is not $\phi(Z)$-relevant. Let $S^* \subseteq Z$ be a smallest set which is $Z$-relevant but not $\phi(Z)$-relevant and let $d = |S^*|$. This is the value of $d$ for which we will prove parts (1) and (2). In order to do this, it suffices to show the following: for any $\phi(Z)$-relevant set $T\subseteq \phi(Z)$ with $|T|\leq d$, it must be the case that $T\subseteq Z$ and in fact, that $T$ is $Z$-relevant. This immediately implies (2), and since $S^*$ is a $Z$-relevant set which is not $\phi(Z)$-relevant, we also get (1). 

Suppose $T\subseteq \phi(Z)$ is a $\phi(Z)$-relevant set of size $|T|\leq d$. Suppose for contradiction that either (a) $T\not \subseteq Z$, or (b) $T\subseteq Z$ but $T$ is not $Z$-relevant. In either case, since $\Gamma_{\phi(Z)}(T)$ is non-empty, it follows that $\Gamma_{Z}(T\cap Z)$ is non-empty as well, so that there exists $T'\subseteq T\cap Z$ such that $T'$ is $Z$-relevant. In case (a), $|T\cap Z| < |T| \leq d$ and in case (b) $T'\subsetneq T$, so that in either case, $|T'| \leq d-1$. Since $S^*$ is a smallest $Z$-relevant set which is not $\phi(Z)$-relevant, it must be the case that $T'$ is $\phi(Z)$-relevant. However, since $T'\subseteq T$, this contradicts that $T$ is $\phi(Z)$-relevant.
\end{claimproof}

Recall that we wish to show that $|Y| = |\phi^*(\emptyset)| = O_{r,m}(1)$. For nonnegative integers $N_{0},\dots,N_{r},z$, define $F(N_{0},\dots,N_{r},z)$ to be the maximum of $|\phi^{*}(Z)|$, over all $Z\subseteq V$ with
\[|Z|=z,\quad N_{0}(Z)\le N_{0},\quad N_{1}(Z)\le N_{1},\quad\dots,\quad N_{r}(Z)\le N_{r},\]
and over all $r$-uniform hypergraphs $G$ on \emph{any} vertex set $V$. (\emph{A priori}, it is possible that no such maximum exists, in which case we set $F(N_0,\dots,N_r,z)=\infty$.) Since $|Y| = |\phi^*(\emptyset)| \leq F(1,0,\dots,0)$, our goal is to show that $F(1,0,\dots,0) < \infty$. Since the input parameters to the function $F$ are only $r$ and $m$ (via the definition of $\phi^*(Z)$), in this case it is clear that $F(1,0,\dots,0)$ depends only on $r$ and $m$. In fact, writing $\mb N_{\ge 0}$ for the nonnegative integers, we will show that for any $(N_0,\dots, N_r, z) \in \mb{N}_{\ge 0}^{r+2}$, we have $F(N_0,\dots, N_r, z) < \infty$.

To see this, note that the above claim implies a recurrence for $F(N_0,\dots, N_r, z)$: for all $(N_0,\dots, N_r, z)$, either $F(N_0,\dots, N_r, z)=z$ or
\begin{equation}F(N_0,\dots, N_r, z) \le \max_{d\in \{0,\dots,r\}:N_d> 0} F(N_0,\dots, N_{d-1}, N_d-1, (z+rm)^{d+1},\dots, (z+rm)^r, z+rm)\label{eq:recurrence}\end{equation}
 (here we use the convention that the maximum of the empty set is $-\infty$; in other words, $F(0,\dots,0,z)=z$, which is also easy to see directly). Now, the desired result follows by induction (most easily described in a ``transfinite'' way): since the lexicographic order on $\mb{N}_{\ge 0}^{r+2}$ is a well-order, if $F(N_0,\dots, N_r, z) = \infty$ for some $(N_0,\dots, N_r, z)$ then there must exist a lexicographically \emph{minimal} $(N_0^*,\dots, N_r^*, z^*)$ for which $F(N_0^*,\dots, N_r^*, z^*) = \infty$. But this is impossible: we would have $F(N_0^*,\dots, N_r^*, z^*) \ne z^*$, so \cref{eq:recurrence} would contradict lexicographic minimality.

Finally, the ``moreover'' part is clear by construction.
\end{proof}

\section{A variance bound for polynomials on the slice}\label{sec:variance}
We need one more technical ingredient for the proof of \cref{conj:1/e}, namely a 
bound on the variance of a polynomial on the slice.
\begin{proposition}\label{lem:cheap-variance}
    For any $n\ge k$, let $\lambda\in\mb R[x_1,\dots,x_n]$ be an $n$-variable multilinear polynomial with degree at most $d$ whose coefficients all have absolute value at most $q$. Let $\vec \sigma\sim \on{Slice}(n,k)$. Then
    \[\on{Var}[\lambda(\vec \sigma)]\lesssim_{d,q} n^{2d-1}.\]
\end{proposition}
\begin{proof}
Since $x_1 + \dots + x_n$ is constant on $\vec{x} \in \on{Slice}(n,k)$, it follows that
$\on{Var}[\lambda(\vec{\sigma})] = \on{Var}[Q(\vec{\sigma})]$, where $Q(\vec{x}) = \lambda(\vec{x}) + q(x_1 + \dots + x_n)^d$.
Since the coefficients of $\lambda$ have absolute value at most $q$, it follows that the coefficients of $Q$ are non-negative and have absolute value at most $O_d(q)$. The key point is the following: let $W, T \subseteq [n]$  be disjoint sets with $|W| = i$ and $|T| = j$. Then, writing $a\equiv b$ if $a$ and $b$ have the same sign, we have
\begin{align*}
\on{Cov}[\vec \sigma^W, \vec \sigma^T] &=
    \frac{\binom{n-i-j}{k-i-j}}{\binom{n}{k}} - \frac{\binom{n-i}{k-i}}{\binom{n}{k}}\cdot \frac{\binom{n-j}{k-j}}{\binom{n}{k}}\\
    &\equiv \binom{n-i-j}{k-i-j}\cdot \binom{n}{k} - \binom{n-i}{k-i}\cdot \binom{n-j}{k-j}\\
    &\equiv  \binom{n}{i}\cdot \binom{k-j}{i} - \binom{n-j}{i}\cdot \binom{k}{i}\\
    &\leq 0,
\end{align*}
where the last inequality follows since \[\frac{\binom{n}{i}\cdot\binom{k-j}{i}}{\binom{n-j}{i}\cdot\binom{k}{i}}=\prod_{s=0}^{i-1} \frac{(n-s)(k-j-s)}{(n-j-s)(k-s)} \le 1\] for $k \leq n$.
Note also that for any $W, T\subseteq [n]$, we have $\on{Cov}[\vec \sigma^W, \vec \sigma^T] \leq \mb{E}[\vec \sigma^W \vec \sigma^T] \leq 1$. 

Recall that the coefficients $\wh{Q}(S)$ are non-negative of size at most $O_d(q)$. Moreover, $\wh{Q}(S) = 0$ if $|S| > d$.  Putting everything together, we have
\begin{align*}
    \on{Var}[Q(\vec{\sigma})] &= \sum_{W,T}\wh{Q}(W)\wh{Q}(T)\on{Cov}[\vec \sigma^W, \vec \sigma^T] \leq \sum_{\substack{W, T \subseteq [n]\\ W\cap T \neq \emptyset}} \wh{Q}(W)\wh{Q}(T)\on{Cov}[\vec \sigma^W, \vec \sigma^T]\\
    &\lesssim_{d} \sum_{|W| \leq d} \sum_{\substack{|T|\leq d \\ T\cap W \neq \emptyset}} q^2  \lesssim_{d} n^{d}\cdot n^{d-1}\cdot q^{2}. \qedhere
\end{align*}  
\end{proof}

\section{Completing the proof of the hypergraph edge-statistics conjecture}\label{sec:completing}
Now, we combine all the ingredients collected so far to prove \cref{conj:1/e}. Recall that the notation $\alpha \ll \beta_1,\dots,\beta_q$ (respectively, $\alpha \gg \beta_1,\dots,\beta_q$) means ``$\alpha$ is sufficiently small in terms of $\beta_1,\dots,\beta_q$'' (respectively, ``$\alpha$ is sufficiently large in terms of $\beta_1,\dots,\beta_q$'').
\begin{proof}[Proof of \cref{conj:1/e}]
Recall that we are to prove that if $k\gg r,\eps$ and $\ell\notin\{0,\binom{k}{r}\}$, then $\operatorname{ind}_{r}(k,\ell)\le 1/e+\varepsilon$. 
    Since $\on{ind}_r(k,\ell) = \on{ind}_r(k, \binom{k}{r}-\ell)$, it suffices to assume that  $\ell \leq \lceil \binom{k}{r}/2 \rceil$. Further, it suffices to prove the statement only for (say) $\ell < k^{-1/2}\binom{k}{r}$, since in the complementary regime $k^{-1/2}\binom{k}{r}\leq \ell \leq \lceil\binom{k}{r}/2\rceil$, a stronger statement follows from \cref{conj:dense}.    

    So, consider integers $k,\ell$ satisfying $0<\ell<k^{-1/2}\binom kr$ and fix any $\eps>0$. 
    The dependence of $k$ on $r,\varepsilon$ will be moderated by additional parameters $R, m, q, \delta$, which will play a role later in the argument. Specifically, we first need $R \gg \eps$, then $m \gg R$, then $q \gg m,r$ and then $\delta \ll q$. Finally, $k \gg m,q,r,R,\delta,\varepsilon$. To summarise, the relative sizes of various parameters should be thought of as
    \[k\gg1/\delta\gg q\gg m\gg R\gg 1/\varepsilon.\]
    
    Let $n=Rk$, let $G$ be an $r$-uniform hypergraph on the vertex set $[n]$, and let $U$ be a random subset of $k$ vertices of $G$. As $N_r(n, k, \ell)/\binom nk$ is nonincreasing in $n$, it suffices to prove that
\[\Pr[e(G[U])=\ell]\le 1/e+\varepsilon.\]

Let $Y$ be the set obtained by applying \cref{lem:vertex-cover} to $G$ (with our value of $m$).

\medskip
\paragraph{\bf Case I: $Y\neq \emptyset$} Consider the random variable
$\mb E\big[e(G[U])\,\big|\,Y\cap U\big]$. This can be interpreted as a multilinear polynomial evaluated at $\vec \sigma\sim\on{Slice}(n,k)$ (where $\sigma_i = \mbm{1}_{i \in U}$), that only depends on $|Y|=O_{m,r}(1)$ of its variables. Indeed, we have 
\begin{align*}
\mb E\big[e(G[U])\,\big|\,Y\cap U\big]=
\sum_{W \in E(G)}\Pr\big[W \in E(G[U])\,\big|\, Y \cap U\big] &= \sum_{W \in E(G)}\Pr[W\setminus Y \subseteq U]\,\mbm{1}_{W\cap Y \subseteq U \cap Y}\\
&= \sum_{W\in E(G)}\Pr[W\setminus Y\subseteq U]\,\vec \sigma^{W\cap Y}
\end{align*}
where $E(G)$ denotes the set of edges of $G$ (cf.\ the expression for $e(G[U])$ at the start of \cref{sec:slice-coupling}). Note that this polynomial has nonnegative coefficients. Also, since we are assuming $Y\neq \emptyset$, every edge of $G$ intersects $Y$, so the constant coefficient of this polynomial is zero. 
So, by \cref{thm:Poisson-anticoncentration,thm:slice-vs-product} (with $\yvar=|Y|$ and $\gamma=\varepsilon/2$ and $t=3^{-|Y|}\ell$), since $R,k \gg \eps$, it follows that except with probability $1/e+\varepsilon/2$, 
\begin{equation}
    \Big|\mb E[e(G[U])\,|\,Y\cap U]-\ell\Big|> 3^{-|Y|}\ell\gtrsim_{r,m}\ell.\label{eq:exclusion}
\end{equation}

Condition on any outcome of $Y\cap U$ such that \cref{eq:exclusion} holds. The remaining randomness is comprised of the random set $U\setminus Y$ (which is a uniformly random subset of $[n]\setminus Y$ of size $k-|Y\cap U|$).

Recall from \cref{def:G-YX} that $G_{Y}(X)=\{e\backslash X:e\in G-(Y\backslash X),\;e\backslash X\ne\emptyset\}$. If $G_Y(Y\cap U)=\emptyset$, then we are done: given our conditioning, $e(G[U])$ would then take some value with probability 1, and this value cannot be equal to $\ell$ since we are working with an outcome of $Y\cap U$ for which \cref{eq:exclusion} holds. Therefore, we can assume that $G_Y(Y\cap U)$ is non-empty. In this case, given our conditioning, we can write $e(G[U])=\lambda(\vec \sigma)$, where $\vec \sigma \sim \on{Slice}(n - |Y|, k - |Y\cap U|)$ and $\lambda$ is a multilinear polynomial of some degree $d\in[r-1]$. Note that the coefficients of $\lambda$ all lie in the set $\{0,1,\dots,q\}$ (here we are using that $q\gg m,r$, so $q\ge\binom{|Y|}r$). Also, by the definition of $Y$, we have $\nu(H^{d}_0(\lambda))=\nu(G_Y(Y\cap U))\geq m$.
Our objective is to show that $\Pr[\lambda(\vec \sigma)=\ell]\le \varepsilon/2$.

Recall that $\delta \ll r,q$ and $m \gg R \gg \eps$ and $k\gg r,q,m$. By
\cref{lem:sparse-LO}, at least one of the following holds:
\begin{enumerate}
\item $\lambda$ has at least $\delta n^{d}$ nonzero degree-$d$ coefficients, or
\item we have
\[
\sup_{\ell\in\mb R}\Pr[\lambda(\vec{\sigma})=\ell]\lesssim_{d,q} \frac1{m^{1/3}}\le \varepsilon/2.
\]
\end{enumerate}
In case (2) we are done. In case (1), we have $\mb E[\lambda(\vec \sigma)]\gtrsim \delta k^d\gtrsim_{\delta} k^d$ while $\on{Var}[\lambda(\vec \sigma)]\lesssim_{d,q} n^{2d-1}\lesssim_{d,q, R} k^{2d-1}$ by \cref{lem:cheap-variance}. So, by Chebyshev's inequality, except with probability at most $\varepsilon/2$ we have
\[\Big|\lambda(\vec \sigma)-\mb E[\lambda(\vec \sigma)]\Big|\lesssim_{d,q,R,\delta,\varepsilon} k^{-1/2}\mb E[\lambda(\vec \sigma)].\]
Recalling from \cref{eq:exclusion} that $\mb E[\lambda(\vec \sigma)]$ is excluded from a range of the form $(1\pm \Omega_{r,m}(1))\ell$, the desired result follows for $k \gg m,q,r,R,\delta,\varepsilon$. 

\medskip
\paragraph{\bf Case II: $Y = \emptyset$} If $G$ is empty, then we are done, since $\ell \neq 0$ by assumption. Otherwise, we can write $e(G[U])=\lambda_G(\vec \sigma)$, where $\lambda_G$ has degree $r$ and $\vec{\sigma} \sim \on{Slice}(n,k)$. In this case we will be able to prove the stronger bound
\[\Pr[e(G[U])=\ell]=\Pr[\lambda_G(\vec \sigma)=\ell]\le \varepsilon/2\le 1/e+\varepsilon.\]
As in Case I above, we have $\nu(H^{(r)}_0(\lambda)) \geq m$, and similarly arguing via \cref{lem:sparse-LO}, we only need to consider the case that $G$ has at least $\delta n^{r}$ nonzero degree-$r$ coefficients. By the same argument as above, except with probability at most $\eps/2$ we have
\[\Big|\lambda_G(\vec \sigma)-\mb E[\lambda_G(\vec \sigma)]\Big|\lesssim_{d,q,R,\delta,\varepsilon} k^{-1/2}\mb E[G(\vec \sigma)].\]
The desired conclusion now follows for $k \gg q,r,R,\delta,\varepsilon$, since $\mb{E}[\lambda_G(\vec{\sigma})] \gtrsim_{R,\delta}k^{r}$, whereas by assumption $\ell < k^{-1/2}\binom{k}{r} \le k^{r-1/2}$. 
\end{proof}

\bibliographystyle{amsplain_initials_nobysame_nomr}
\bibliography{main}

\providecommand{\bysame}{\leavevmode\hbox to3em{\hrulefill}\thinspace}
\providecommand{\MR}{\relax\ifhmode\unskip\space\fi MR }
\providecommand{\MRhref}[2]{%
  \href{http://www.ams.org/mathscinet-getitem?mr=#1}{#2}
}
\providecommand{\href}[2]{#2}
\begin{thebibliography}{10}

\bibitem{AHKT20}
N.~Alon, D.~Hefetz, M.~Krivelevich, and M.~Tyomkyn, \emph{Edge-statistics on
  large graphs}, Combin. Probab. Comput. \textbf{29} (2020), no.~2, 163--189.

\bibitem{BHLF}
J.~Balogh, P.~Hu, B.~Lidick\'y, and F.~Pfender, \emph{Maximum density of
  induced 5-cycle is achieved by an iterated blow-up of 5-cycle}, European J.
  Combin. \textbf{52} (2016), 47--58.

\bibitem{BS15}
B.~Bollob\'{a}s and A.~Scott, \emph{Intersections of hypergraphs}, J. Combin.
  Theory Ser. B \textbf{110} (2015), 180--208.

\bibitem{Cos13}
K.~P. Costello, \emph{Bilinear and quadratic variants on the
  {L}ittlewood-{O}fford problem}, Israel J. Math. \textbf{194} (2013), no.~1,
  359--394.

\bibitem{CTV06}
K.~P. Costello, T.~Tao, and V.~Vu, \emph{Random symmetric matrices are almost
  surely nonsingular}, Duke Math. J. \textbf{135} (2006), no.~2, 395--413.

\bibitem{Ehm91}
W.~Ehm, \emph{Binomial approximation to the {P}oisson binomial distribution},
  Statist. Probab. Lett. \textbf{11} (1991), no.~1, 7--16.

\bibitem{Erd45}
P.~Erd\H{o}s, \emph{On a lemma of {L}ittlewood and {O}fford}, Bull. Amer. Math.
  Soc. \textbf{51} (1945), 898--902.

\bibitem{FKS21}
J.~Fox, M.~Kwan, and L.~Sauermann, \emph{Combinatorial anti-concentration
  inequalities, with applications}, Math. Proc. Cambridge Philos. Soc.
  \textbf{171} (2021), no.~2, 227--248.

\bibitem{FS20}
J.~Fox and L.~Sauermann, \emph{A completion of the proof of the edge-statistics
  conjecture}, Adv. Comb. (2020), Paper No. 4, 52.

\bibitem{FSW21}
J.~Fox, L.~Sauermann, and F.~Wei, \emph{On the inducibility problem for random
  {C}ayley graphs of abelian groups with a few deleted vertices}, Random
  Structures Algorithms \textbf{59} (2021), no.~4, 554--615.

\bibitem{Hir14}
J.~Hirst, \emph{The inducibility of graphs on four vertices}, J. Graph Theory
  \textbf{75} (2014), no.~3, 231--243.

\bibitem{Kan14}
D.~M. Kane, \emph{The correct exponent for the {G}otsman-{L}inial conjecture},
  Comput. Complexity \textbf{23} (2014), no.~2, 151--175.

\bibitem{KS}
M.~Kwan and L.~Sauermann, \emph{Resolution of the quadratic
  {L}ittlewood--{O}fford problem}, arXiv preprint 2312.13826.

\bibitem{KST19}
M.~Kwan, B.~Sudakov, and T.~Tran, \emph{Anticoncentration for subgraph
  statistics}, J. Lond. Math. Soc. (2) \textbf{99} (2019), no.~3, 757--777.

\bibitem{LO43}
J.~E. Littlewood and A.~C. Offord, \emph{On the number of real roots of a
  random algebraic equation. {III}}, Rec. Math. [Mat. Sbornik] N.S.
  \textbf{12(54)} (1943), 277--286.

\bibitem{LMR23}
X.~Liu, D.~Mubayi, and C.~Reiher, \emph{The feasible region of induced graphs},
  J. Combin. Theory Ser. B \textbf{158} (2023), no.~part 2, 105--135.

\bibitem{MMNT19}
A.~Martinsson, F.~Mousset, A.~Noever, and M.~Truji\'{c}, \emph{The
  edge-statistics conjecture for {$\ell\ll k^{6/5}$}}, Israel J. Math.
  \textbf{234} (2019), no.~2, 677--690.

\bibitem{MNV16}
R.~Meka, O.~Nguyen, and V.~Vu, \emph{Anti-concentration for polynomials of
  independent random variables}, Theory Comput. \textbf{12} (2016), Paper No.
  11, 16.

\bibitem{PG75}
N.~Pippenger and M.~C. Golumbic, \emph{The inducibility of graphs}, J.
  Combinatorial Theory Ser. B \textbf{19} (1975), no.~3, 189--203.

\bibitem{RV13}
A.~Razborov and E.~Viola, \emph{Real advantage}, ACM Trans. Comput. Theory
  \textbf{5} (2013), no.~4, Art. 17, 8.

\bibitem{RS96}
J.~Rosi\'{n}ski and G.~Samorodnitsky, \emph{Symmetrization and concentration
  inequalities for multilinear forms with applications to zero-one laws for
  {L}\'{e}vy chaos}, Ann. Probab. \textbf{24} (1996), no.~1, 422--437.

\bibitem{Uel}
R.~Ueltzen, \emph{Characterizing graphs with high inducibility}, arXiv preprint
  2411.17362.

\end{thebibliography}

\end{document}